%% file: pmelons.tex
\documentclass[10pt,a4paper]{amsart}

\usepackage{amssymb,amsmath,amsthm} 
\usepackage{rotating,graphicx,psfrag,epsfig}
\usepackage{cite}

\theoremstyle{definition}
\newtheorem{lemma}{Lemma}

\newtheorem*{definition*}{Definition}

\newtheorem{theorem}{Theorem}
\newtheorem{corollary}{Corollary}
\newtheorem{remark}{Remark}
\newtheorem{proposition}{Proposition}

\newcommand{\Z}{\mathbb{Z}}
\newcommand{\C}{\mathbb{C}}
\newcommand{\N}{\mathbb{N}}
\newcommand{\R}{\mathbb{R}}

\DeclareMathOperator*{\Res}{Res}
\renewcommand{\P}{\mathbb{P}}

\newcommand{\sgn}{\mathrm{sgn\,}}

\newcommand{\E}{\mathrm{\mathbb E}}

\newcommand{\floor}[1]{\ensuremath{\left\lfloor #1 \right\rfloor}}

\renewcommand{\vec}[1]{ \ensuremath{ {\bf#1}} }

\title[The height of Watermelons with wall]{The height of Watermelons with wall}
\author[Thomas Feierl]{Thomas Feierl$^\dagger$}
\address{Thomas Feierl \\ Fakult\"at f\"ur Mathematik \\ Universit\"at Wien \\ Nordbergstr. 15 \\ 1090 Wien \\ Austria}
\thanks{$^\dagger$ Research supported by the Austrian Science Foundation FWF, grant S9607-N13}
\date{February 19, 2008} 

\begin{document}

\begin{abstract}
		We derive asymptotics for the moments as well as the weak limit of the height distribution of
		watermelons with $p$ branches with wall.
		This generalises a famous result of de Bruijn, Knuth and Rice~\cite{MR0505710} on the average
		height of planted plane trees, and results by Fulmek~\cite{Fulmek} and Katori et al.~\cite{katori}
		on the expected value, respectively higher moments, of the height distribution of watermelons with two branches.

		The asymptotics for the moments depend on the analytic behaviour of certain multidimensional Dirichlet series. In order to
		obtain this information we prove a reciprocity relation satisfied by the derivatives of one of
		Jacobi's theta functions, which generalises the well known reciprocity law for Jacobi's theta functions.
\end{abstract}

\maketitle
\section{Introduction}

The model of \emph{vicious walkers} was introduced by Fisher~\cite{MR751710}. He gave a number of
applications in physics, such as modelling wetting and melting processes
(for more applications in physics such as polymer networks and fibrous structures, we refer the reader to \cite{MR1384738,schehr} and references therein).

In general, the model of vicious walkers is concerned with $p$ random walkers on a $d$-dimensional lattice.
In the lock step model, at each time step all of the walkers move one step in any of the allowed directions, such that at no time any two random walkers share the same lattice point.

A configuration that attracted much interest amongst mathematical physicists and combinatorialists is the \emph{watermelon configuration},
which is a special case of the one dimensional vicious walker model on the integer lattice.
In this configuration, the walkers are initially positioned at the points $0,2,\dots,2p-2$.
At each time step, the particles may simultaneously move one step to the left or to the right on the lattice subject to the model restrictions.
After a fixed number of time steps, $2n$ say, the walkers are required to return to their starting points.
This configuration can be studied with or without presence of an impenetrable wall, and with or without deviation.
In case of presence of such a wall, the integer lattice is replaced by the lattice of non-negative integers.
The interpretation being that the wall is located at position zero (or, more precisely, right below zero).
Particles are allowed to sit ``at the wall'' but may not jump over the wall so that no particle will ever visit a negative integer lattice site.
Tracing the paths of the vicious walkers on the lattice through time, we obtain a set of lattice paths with certain properties.
The precise properties are given in the definition below, serving as our definition of the watermelon configuration with wall restriction underlying this paper.
\begin{definition*}
	Consider the lattice in $\R^2$ spanned by the two vectors $(1,1)$ and $(1,-1)$.
	A \emph{$p$-watermelon of length $2n$ with wall restriction} is a set of $p$ lattice walks, each of which consists of $2n$ steps of the form $(1,\pm 1)$, such that
	\begin{enumerate}
		\item for $j=1,\dots,p$, the $j$-th path starts at $(0,2j-2)$ and ends at $(2n,2j-2)$,
		\item the lattice walks are pairwise vertex disjoint, and
		\item no lattice walk passes below the horizontal line passing through the origin.
	\end{enumerate}
\end{definition*}
In the above definition, Condition~(3) corresponds to the wall restriction, whereas Condition~(2) does not capture our ``vicious constraint'' by itself.
Indeed, in order to rule out particles jumping over others, we need as an additional ingredient from Condition~(1) the fact that the particles are initially positioned at sites all of which have the same parity.
An example illustrating this definition is depicted in Figure~\ref{fig:wm_with_wall_ex}, where, for the moment, the broken line labelled $13$ should be ignored.
%
%

\begin{figure}[width=8cm]
		\begin{center}
				\input{watermelons_fig.tex}
		\end{center}
		\caption{Example of a $4$-watermelon of length 18 with wall and height $13$}
		\label{fig:wm_with_wall_ex}
\end{figure}
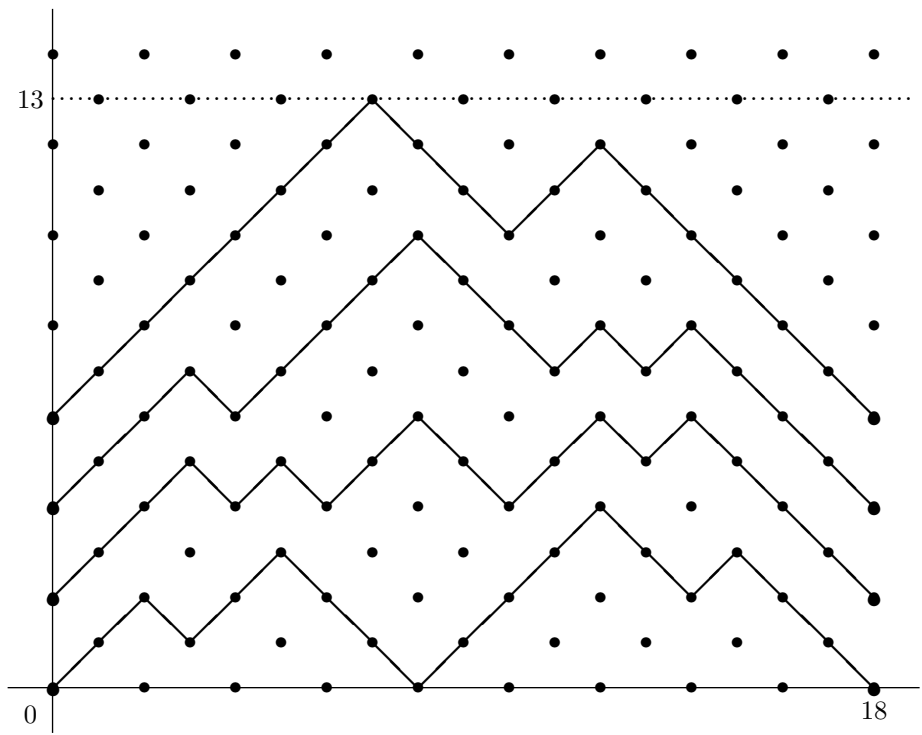

Since Fisher's introduction~\cite{MR751710} of the vicious walkers model numerous papers on this 
subject have appeared. While early results mostly analyse vicious walkers in a continuum limit,
there are nowadays many  results for certain configurations directly based on the lattice path description above.
For example, Guttmann, Owczarek and Viennot~\cite{MR1651492} related the star and watermelon configurations
to the theory of Young tableaux and integer partitions, and re-derived results for the total number
of stars and watermelons without wall. Later, Krattenthaler, Guttmann and Viennot~\cite{MR1801472}
proved new, exact as well as asymptotic, results for the number of certain vicious walkers with wall.
Recently, Krattenthaler~\cite{Kratt_Wall_Contact} analysed the number of contacts of the bottom most walker in
the case of watermelons with wall,
continuing earlier work by Brak, Essam and Owczarek~\cite{MR1832080}.

In 2003, Bonichon and Mosbah~\cite{MR2022577} presented an algorithm for uniform random generation of watermelons,
which is based on the counting results by Krattenthaler, Guttmann and Viennot~\cite{MR1801472}
(see Theorems 1 and 6 therein).
Amongst other things they used their generator for obtaining experimental results on the \emph{height of watermelons}.
\begin{definition*}
	The \emph{height of a $p$-watermelon} is the smallest number $h$ such that the upper most branch does not cross the horizontal line $y=h$.
\end{definition*}
An illustration of this definition is given in Figure~\ref{fig:wm_with_wall_ex}, depicting a $4$-watermelon of height $13$ (indicated by the horizontal dashed line).

The parameter ``height'' has attracted the interest of quite a few people.
It is the purpose of the following two paragraphs to provide the reader with a rough overview of known results on and related to the central question studied in this manuscript.

For $p=1$, our watermelon definition reduces to the definition of so called Dyck paths.
A fact that is also nicely illustrated by the bottom most path in Figure~\ref{fig:wm_with_wall_ex}.
It is well-known that these Dyck paths are in bijection (left-order-traversal) with planted plane trees,
and that under this bijection the height of a Dyck path corresponds to the height of the corresponding tree.
The asymptotic behaviour of the average height of planted plane trees was determined by de Bruijn, Knuth and Rice~\cite{MR0505710},
that is, they solved the average height problem for $1$-watermelons with wall.
Recently, Fulmek~\cite{Fulmek} extended their reasoning and determined the asymptotic behaviour of the average height
of $2$-watermelons with wall.
Katori, Izumi and Kobayashi~\cite{katori} considered the diffusion scaling limit of $2$-watermelons, and obtained the
leading asymptotic term for all moments of the height distribution as well as a central limit theorem.
The limiting process of $p$-watermelons has been investigated by Gillet~\cite{gillet}. He succeeded in proving the convergence of (properly scaled) watermelons to a certain limiting process, which he
characterised by a system of stochastic differential equations.
It is interesting to note that this limiting process can be interpreted as a variant of Dyson's Brownian motion model~\cite{MR0148397}.
In~\cite{gillet}, the author also determined the one-dimensional distribution of this limiting process, and proved that it is equal to  the (properly scaled) distribution of the eigenvalues in the $(2p+1)\times(2p+1)$-dimensional Gaussian antisymmetric hermitian matrix ensemble (for definitions of the Gaussian ensembles, we refer the reader to Mehta~\cite{MR2129906}).
For closely related models without the return condition (star configurations), Katori and Tanemura~\cite{MR2029612} proved convergence of the one-dimensional distributions to the eigenvalue distributions of the Gaussian unitary ensemble (with wall restriction) and the Gaussian orthogonal ensemble (without wall restriction).
Tracy and Widom~\cite{MR2326237} studied Gillet's limiting process (Dyson's Brownian excursion) in the limit $p\to\infty$, and give Fredholm determinant expressions for the top most and bottom most branch. As applications, they studied the area under these two branches.

In this paper we rigorously analyse the height of $p$-watermelons of length $2n$ with wall, and obtain asymptotics
for all moments of the height distribution as $n\to \infty$ as well as a central limit theorem. In particular,
we show that the $s$-th moment behaves
like $s\kappa_{s}^{(p)}n^{s/2}-3\binom{s}{2}\kappa_{s-1}^{(p)}n^{(s-1)/2}+O(n^{s/2-1}+n^{p/2-p^2}\log n)$ as $n\to\infty$
for some explicit numbers $\kappa_{s}^{(p)}$, see Theorem~\ref{thm:sth-moment-asymptotics} at the end of
Section~\ref{sec:moments}.
The nature of our result explains the somewhat inconclusive predictions in \cite{MR2022577}.
To be more specific, Bonichon and Mosbah~\cite{MR2022577} predicted, based on numerical experiments, that
$\kappa_1^{(p)}\approx \sqrt{(1.67p-0.06)}$.
Although it does not seem unlikely that the constant $\kappa_{1}^{(p)}$, as given in Theorem~\ref{thm:sth-moment-asymptotics},
behaves like $\sqrt p$ as $p\to \infty$, a rigorous proof is still lacking and work in progress.
		
The proof of our result can be summarised as follows. As a first step, we represent the total number of watermelons
and the number 
with height restriction in terms of certain determinants (see Lemma~\ref{lem:exact_expressions}), the entries being
sums of binomial coefficients. From these determinants we then obtain an exact expression for the $s$-th moment of the height 
distribution. After normalisation we may apply Stirling's formula and obtain an expression that
can be asymptotically evaluated using Mellin transform techniques (see Lemma~\ref{lem:g_a_asymptotics}).
This kind of approach goes back to de Bruijn, Knuth and Rice~\cite{MR0505710}. 
Fulmek~\cite{Fulmek} adopted their approach for the asymptotic analysis of $2$-watermelons. The new objects which
arise here (and, in general, when extending this approach to the asymptotic analysis of $p$-watermelons)
are certain multidimensional Dirichlet series (instead of Riemann's zeta function as in \cite{MR0505710}).
An additional complication with which one has to cope is the increasing number of cancellations of leading
asymptotic terms that one encounters in the calculations while the number $p$ of branches becomes
bigger. Thus, while a brute force approach will eventually produce a result for any fixed $p$
(this is, in essence, what Fulmek~\cite{Fulmek} and Katori et al.~\cite{katori} do for $p=2$), the main 
difficulty that we have to overcome in order to arrive
at an asymptotic result for {\it arbitrary} $p$ is to trace the roots of these cancellations. We
accomplish this with the help of Lemma~\ref{lem:wm_with_wall_asymptotic_for_sum}. It allows us to exactly pin down
which cancellations take place and to extract explicit formulae for the first two terms which survive the cancellations.
The multidimensional Dirichlet series which arise in our analysis are the subject of the subsequent section.
What we need is information on their poles. This information is obtained with the help of a relation that
generalises the reciprocity law for Jacobi's theta functions (see Equation~(\ref{eq:theta_derivative_reciprocity})),
that is proved in Proposition~\ref{lem:theta_derivative_reciprocity}.
We note that our definition of these Dirichlet series differs slightly from Fulmek's definition, which makes
the analysis somewhat easier. 
These Dirichlet series that we encounter are related to so-called twisted multivariate zeta functions, studied, e.g., by
de Crisenoy~\cite{MR2278751} and de Crisenoy and Essouabri~\cite{deCrisenoy}. However, their results cannot be used since they
do not apply to our multidimensional Dirichlet series, which are explicitly excluded in these two papers.
They can also be found within a class of multidimensional Dirichlet series studied by 
Cassou-Nogu{\`e}s~\cite{MR692105}. In principle we could apply her results to our Dirichlet series and would obtain information
on the poles of the analytic continuation of these series. But this would be cumbersome, and in our case it is more
straightforward to obtain this information using the generalised reciprocity relation (see end of Section~\ref{sec:jacobi}), which we are going to need in the proof of Theorem~\ref{thm:sth-moment-asymptotics} anyway.

We mention that small modifications immediately yield analogous results for $p$-watermelons with a horizontal wall
positioned 
at some negative integer. Also, the analysis of the height distribution of watermelons {\it without wall} can be 
accomplished in a completely analogous fashion (see~\cite{watermelons:withoutwall}, and also \cite{katori2,schehr} for results in the continuous setting).

The paper is organised as follows.
In Section~\ref{sec:main results}, we state and discuss the main results of this manuscript.
Section~\ref{sec:jacobi} is devoted to the study of certain multidimensional Dirichlet series and related exponential sums that is crucial to the proof of our theorems.
In this section we also give our reciprocity relation generalising Jacobi's reciprocity law (see Proposition~\ref{lem:theta_derivative_reciprocity}.
The last two sections contain the proofs of Theorem~\ref{thm:sth-moment-asymptotics} and Theorem~\ref{thm:centrallimitlaw}, respectively.

We close this section by fixing some notation. Vectors are denoted using bold face letters and are assumed
to be $p$-dimensional row vectors. Further, we make use of the $1$-norm and the $2$-norm of vectors,
viz. $|\vec w|_1=w_0+\cdots+w_{p-1}$ and $|\vec w|_2^2=w_0^2+\cdots w_{p-1}^2$. Finally, we define
$\vec v^\vec w=v_0^{w_0}\dots v_{p-1}^{w_{p-1}}$. The relation $\vec v \ge \vec w$ is to be understood component-wise.

\section{Main results}
\label{sec:main results}
In this section, we state the main results of this manuscript, give some comments and discuss related results.
We start by fixing some notation.
Let us denote by $M_{2n,h}^{(p)}$ the number of $p$-watermelons with wall with length $2n$ and height strictly smaller than $h$.
Further, we write $M_{2n}^{(p)}$ for the total number of $p$-watermelons with length $2n$. Note that $M_{2n}^{(p)}=M_{2n,h}^{(p)}$
for $h\ge n+2p-1$ and $M_{2n,h}^{(p)}=0$ for $h< 2p$.

Now, let $\mathfrak W_{n}^{(p)}$ denote the set of $p$-watermelons of length $2n$, and let $\P$ denote
the uniform probability measures on these sets, and let $H_{n,p}$ denote the random variable ``height''
on the probability space $\left( \mathfrak W_{n}^{(p)},2^{\mathfrak W_{n}^{(p)}},\P \right)$.
The main results stated below consist of asymptotics for the quantities
\[
\E \left(H_{n,p}\right)^s,\ s\in\N,
\qquad
\textrm{and}
\qquad
\P\left\{ n^{-1/2}H_{n,p}\le t \right\},
\]
where $\E$ denotes the expectation with respect to $\P$.

All asymptotic results involve the infinite sums
\[
\vartheta_{2a}(t)=\sum_{n=-\infty}^{\infty}(-4\pi^2n^2)^{a}e^{-n^2\pi t},
\qquad
a\in\N.
\]
These sums are seen to be proportional to derivatives of a variant of Jacobi's theta function, and come from determining asymptotics for certain sums of binomial coefficients.
For details, we refer to the proof of Theorem~\ref{thm:sth-moment-asymptotics}.
The study of these sums and -- via Mellin transform -- related Dirichlet series is fundamental for the derivation of our main results and is the content of Section~\ref{sec:jacobi}.

\begin{theorem}
		Set $M_p = 2^{p^2}\prod_{i=0}^{p-1}(2i+1)!$ 
		and 
		$T_p(t)=\det\limits_{0\le i,j<p}\left( \vartheta_{2i+2j+2}(t) \right)$.
		For $s\in\N$, the $s$-th moment of the height distribution of $p$-watermelons with wall
		satisfies
		\begin{equation}
				\E H_{n,p}^s = s\kappa_s^{(p)}n^{s/2}-3\binom{s}{2}\kappa_{s-1}^{(p)}n^{(s-1)/2}-\frac{3}{2}
				+O\left( n^{s/2-1}+n^{p/2-p^2}\log n \right)
				\label{eq:sth-moment-asymptotics}
		\end{equation}
		as $n\to\infty$, where 
		\[
		\kappa_{s}^{(p)} = \frac{\pi^{s/2}}{2}\int_{0}^{\infty}t^{-1-s/2}
		\left( 1-\frac{t^{p^2+p/2}}{(-\pi)^{p^2}}\frac{T_p(t)}{M_p} \right) dt,\qquad s>0.
		\]
		\label{thm:sth-moment-asymptotics}
\end{theorem}

This theorem is seen to be valid even for $s\in\C$ with $s$ having positive real part.
For details we refer the reader to Remark~\ref{rem:complex_moments}, where we give some details on where to change the proof of Theorem~\ref{thm:sth-moment-asymptotics} as presented in this manuscript.

\begin{table}
\label{tab:moment_dominant_coef}
\caption{
This table gives numerical approximations for $s\kappa_s^{(p)}$ for small values of $s$ and $p$.
The quantity $s\kappa_s^{(p)}$ is the coefficient of the dominant part of the asymptotics
for the $s$-th moment of the height of $p$-watermelons (see  Theorem~\ref{thm:sth-moment-asymptotics}).
The calculations have been carried out using the integral representation for $\kappa_s^{(p)}$ as given in
Theorem~\ref{thm:sth-moment-asymptotics}. The results shown here conform
well with with numerical results obtained by Fulmek~\cite{Fulmek} and Katori et al.~\cite{katori}.
}
\begin{tabular}{c|r|r|r|r}
		$s\kappa_{s}^{(p)}$ & \centering{$s=1$} & $s=2$ & $s=3$ & $s=4$ \\
		\hline
		$p=1$ & $\sqrt\pi$ & $3.289\dots$ & $6.391\dots$ & $12.987\dots$\\
		$p=2$ & $2.577\dots$ & $6.790\dots$ & $18.282\dots$ & $50.306\dots$ \\
		$p=3$ & $3.207\dots$ & $10.429\dots$ & $34.371\dots$ & $114.817\dots$ \\
		$p=4$ & $3.742\dots$ & $14.141\dots$ & $53.939\dots$ & $207.712\dots$\\
		$p=5$ & $4.215\dots$ & $17.898\dots$ & $76.536\dots$ & 329.655\dots
\end{tabular}
\end{table}

Some numerical approximations for the coefficient of the dominant term of the asymptotics proved in Theorem~\ref{thm:sth-moment-asymptotics} are shown in Table~\ref{tab:moment_dominant_coef}. But Theorem~\ref{thm:sth-moment-asymptotics} does not only give the dominant term of the asymptotics of the $s$-th moment of the height distribution of $p$-watermelons but also the second order term. So, for example,
we obtain the more precise asymptotics
\begin{align*}
		\mathfrak \E H_{n,1} &= \sqrt{\pi n} - \frac{3}{2} + O\left( n^{-1/2}\log n \right), \qquad n\to\infty, \\
		\mathfrak \E H_{n,2} &=  2.577\ldots \sqrt{n} -\frac{3}{2} +O\left( n^{-1/2} \right), \qquad n\to\infty, \\
		\mathfrak \E H_{n,2}^2 &=  6.790\ldots n -3.866\dots\sqrt n +O\left( 1 \right), \qquad n\to\infty.
\end{align*}
The asymptotic result for $\E H_{n,1}$ as stated above originally appeared in~\cite{MR0505710}.

\begin{theorem}
		For $t\in(0,\infty)$ fixed, the random variable $H_{n,p}$ on the set of $p$-watermelons of length $2n$ with wall
		satisfies
	\begin{equation}
			\P\left\{ \frac{H_{n,p}+2}{\sqrt n}\le t\right\} = \frac{\pi^{p/2}t^{-2p^2-p}}{(-2)^{p^2}\prod_{i=0}^{p-1}(2i+1)!}
			\det_{0\le i,j<p}\Big( \vartheta_{2i+2j+2}\left( \frac{\pi}{t^2} \right)\Big)
			+O\left( \frac{1}{nt} \right)
			\label{eq:centrallimitlaw}
	\end{equation}
	as $n\to\infty$,
	where the constant implied by the $O$-term is independent of $t$.
\label{thm:centrallimitlaw}
\end{theorem}
\begin{figure}[]
		\begin{center}
			\includegraphics[width=14cm]{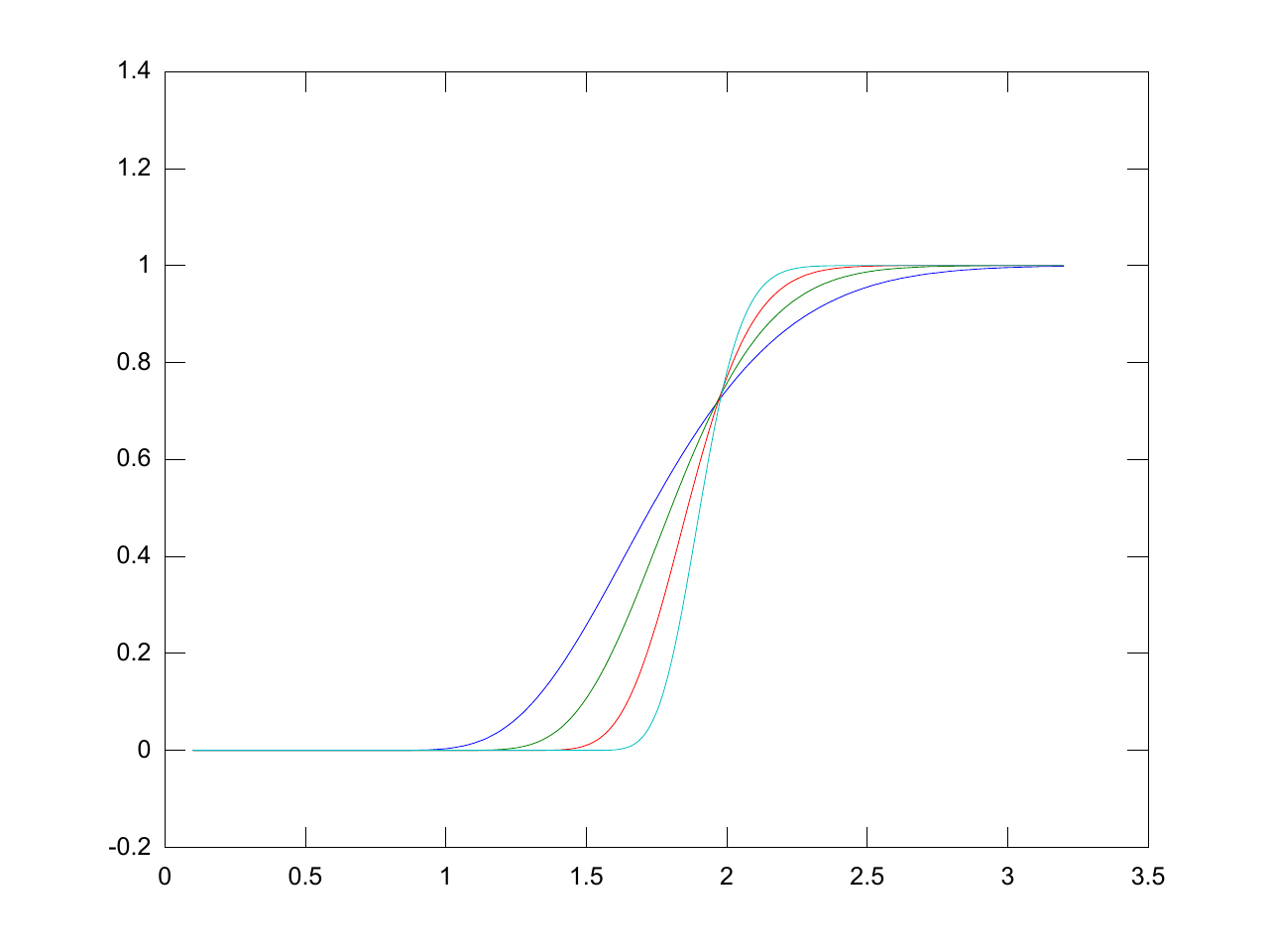}
		\end{center}
		\caption{The plot shows the scaled limiting probability distributions $F_p\left( t\sqrt{p} \right)$ of Theorem~\ref{thm:centrallimitlaw} for $p=1,2,4,8$ for $0\le t\le 3.5$.}
		\label{fig:p->infinity}
		\end{figure}

For the special case $p=1$, Theorem~\ref{thm:centrallimitlaw} reduces to the well known central limit law first proved by 
R{\'e}nyi and Szekeres~\cite{MR0219440}, viz.
\begin{equation}
		\P\left\{ \frac{H_{n,1}}{\sqrt n}\le t \right\} \to
		\sum_{m\in\Z}\left( 1-2(mt)^2\right)e^{-(mt)^2} 
		,\qquad n\to\infty.
		\label{eq:limitlaw_p=1}
\end{equation}
It should be mentioned that this limiting distribution is the distribution function of $\sqrt{2} \max_{0\le x\le 1}e(x)$,
where $e(x)$ denotes the standard Brownian excursion of duration $1$.
For details and references we refer to the survey paper by
Biane, Pitman and Yor~\cite{MR1848256}, in which the authors consider probability
laws related to Brownian motion, Riemann's zeta function and Jacobi's theta functions.

\begin{remark}
	Recall the integral expression for the quantities $\kappa_s^{(p)}$ as given in Theorem~\ref{thm:sth-moment-asymptotics}, viz.
	\[
	\kappa_{s}^{(p)} = \frac{\pi^{s/2}}{2}\int_{0}^{\infty}t^{-1-s/2}
	\left( 1-\frac{t^{p^2+p/2}}{(-\pi)^{p^2}}\frac{T_p(t)}{M_p} \right) dt,\qquad s>0.
	\]
	The change of variables $t\to\pi/t^2$ transforms this integral into
	\[
	\kappa_s^{(p)}=\int_{0}^\infty t^{s-1}\left( 1-f_p(t) \right)dt,
	\]
	where
	\[ 
	f_p(t)=\frac{\pi^{p/2}t^{-2p^2-p}}{(-2)^{p^2}\prod_{i=0}^{p-1}(2i+1)!}
		\det_{0\le i,j<p}\Big( \vartheta_{2i+2j+2}\left( \frac{\pi}{t^2} \right)\Big)
	\]
	is the dominant part of the asymptotics given in Theorem~\ref{thm:centrallimitlaw}.
	It is clear that $f_p(t)$ is a probability distribution function, and furthermore $f_p(t)$ is differentiable.
	Hence, by partial integration, we obtain
	\[ 
	\kappa_s^{(p)}=\int_{0}^\infty t^{s-1}\left( 1-f_p(t) \right)dt = \frac{1}{s}\int_0^\infty t^s f'_p(t)dt.
	\]
	This last integral is simply the $s$-th moment of a random variable with probability distribution $f_p(t)$, precisely as it should be.
\end{remark}
\begin{remark}
	After distribution of the first version of this manuscript on arXiv.org, the asymptotic cumulative distribution function of the random variable ``height'' as given in Theorem~\ref{thm:centrallimitlaw} has been re-derived by two groups.
	Since their expressions differ from the one given by Equation~(\ref{eq:centrallimitlaw}) we want to give some comments on the
	equivalence of these three (more or less) different expressions.
	
	The expression found by  Katori et al.~\cite{katori2} can easily be obtained by an application of the
	reciprocity relation~(\ref{eq:vartheta_a_identity}) to Equation~(\ref{eq:centrallimitlaw}), and therefore, is not essentially different
	from the one given here.

	Schehr et al.~\cite{schehr} expressed the cumulative distribution function of the height as a multiple sum, which can also be easily derived
	from Theorem~\ref{thm:centrallimitlaw}.
	Since $\vartheta_{2a}(t)=\sum_{n=-\infty}^{\infty}(-4\pi^2n^2)^{a}e^{-n^2\pi t}$ we see that the determinant in
	Equation~(\ref{eq:centrallimitlaw}) is equivalent to
	\begin{align*}
		\det_{0\le i,j<p}\Big( \vartheta_{2i+2j+2}\left( \frac{\pi}{t^2} \right)\Big)&=
		\sum_{n_0,\dots,n_{p-1}\in\Z}e^{-\sum_{j=0}^{p-1}(n_j\pi/t)^2}\det_{0\le i,j<p}\left( (-4n_j^2\pi^2)^{i+j+1} \right) \\
		&=(-\pi^2)^{p^2}2^{2p^2+p}
		\sum_{n_0,\dots,n_{p-1}\in\N}\left( \prod_{j=0}^{p-1}n_j^{2j+2} \right)\left( \prod_{0\le i<j<p}(n_j^2-n_i^2) \right)
		e^{-\sum_{j=0}^{p-1}(n_j\pi/t)^2} \\
		&=(-\pi^2)^{p^2}2^{2p^2+p}
		\sum_{1\le n_0<\dots<n_{p-1}}\left( \prod_{0\le i<j<p}(n_j^2-n_i^2) \right)
		\left( \sum_{\sigma\in\mathfrak S_p}\sgn(\sigma)\prod_{j=0}^{p-1}n_{\sigma(j)}^{2j+2}e^{-(n_{\sigma(j)}\pi/t)^2} \right) \\
		&=(-\pi^2)^{p^2}2^{2p^2+p}
		\sum_{1\le n_0<\dots<n_{p-1}}
		\left(  \prod_{0\le i<j<p}(n_j^2-n_i^2)  \right)^2\left( \prod_{j=0}^{p-1}n_j^2e^{-(n_j\pi/t)^2} \right),
 	\end{align*}
	where $\mathfrak S_p$ denotes the set of permutations on the set $\left\{ 0,1,\dots,p-1 \right\}$.
	Here, the second and fourth lines are simple consequences of the Vandermonde determinant formula. For the third line we note that
	the the factor $\sgn(\sigma)$ cancels the sign change of the Vandermonde product due to the rearrangement from the second to the third line.
	Substituting this expression for the determinant in Equation~(\ref{eq:centrallimitlaw}) we arrive at
	\[
	\frac{2^{p^2+p}\pi^{2p^2+p/2}t^{-2p^2-p}}{\prod_{i=0}^{p-1}(2i+1)!}\sum_{1\le n_0<\dots<n_{p-1}}
		\left(  \prod_{0\le i<j<p}(n_j^2-n_i^2)  \right)^2\left( \prod_{j=0}^{p-1}n_j^2e^{-(n_j\pi/t)^2} \right),
	\]
	which is the expression for the cumulative distribution function obtained by Schehr et al~\cite{schehr}.
	It should be noted however, that the derivation in \cite{schehr} is based on physical arguments and path integral techniques and therefore should be considered as being heuristic (at least from a mathematical point of view).
\end{remark}
\begin{remark}
	There is yet another quite interesting representation for the limiting distribution function $F_p(t)$ of Theorem~\ref{thm:centrallimitlaw}.
	Starting from the multiple sum representation of the last remark, where we set $x_j=n_j\pi/(t\sqrt{p})$ we obtain, recalling the Selberg integral evaluation (see, e.g., \cite{MR2129906})
	\[
	\idotsint\limits_{0<x_0<x_1<\dots<x_{p-1}}
	\left( \prod_{0\le i<j<p}(x_j^2-x_i^2) \right)^2\left( \prod_{j=0}^{p-1}x_j^2e^{-x_j^2}dx_j \right) \\
	=
	\frac{\pi^{p/2}}{2^{p^2+p}}\prod_{j=0}^{p-1}(2j+1)!,
	\]
	the representation
	\[
		F_p(t)=
		\frac{
		\left( \frac{\pi }{t\sqrt{p}} \right)^p
		\sum\limits_{\substack{0<x_0<\dots<x_{p-1} \\ x_j\frac{t\sqrt{p}}{\pi}\in\N}}\left( \prod\limits_{0\le i<j<p}(x_j^2-x_i^2) \right)^2\left( \prod\limits_{j=0}^{p-1}x_j^2e^{-px_j^2} \right)
		}{
		\idotsint\limits_{0<x_0<\dots<x_{p-1}}
		\left( \prod\limits_{0\le i<j<p}(x_j^2-x_i^2) \right)^2\left( \prod\limits_{j=0}^{p-1}x_j^2e^{-px_j^2}dx_j \right)
		}
	\]
	Here it should be observed that the multiple sum in the numerator can be interpreted as a Riemann sum approximation to the integral in the denominator.
	\label{rem:F_p(t)-quotient_representation}
\end{remark}

\begin{remark}
	An interesting problem is the behaviour of the limiting distribution $F_p(t)$ of Theorem~\ref{thm:centrallimitlaw} and of the quantities $\kappa_s^{(p)}$ as the number $p$ of paths tends to infinity.
	This question is especially natural from a physical point of view where vicious walkers serve as models for certain gases, the walkers being particles, and one is interested in the behaviour of this gas as the number $p$ of particles tends to infinity.

Based on numerical experiments, Bonichon and Mosbah~\cite{MR2022577} predicted that
\[ \kappa_1^{(p)}\approx\sqrt{1.67p-0.06}\qquad\textrm{as $p\to\infty$.} \]
And in view of Figure~\ref{fig:p->infinity}, one might suspect that $F_p(t\sqrt p)$ should converge to a certain limiting function $F(t)$ as $p\to\infty$, so that (by the change of variables $t\to t/\sqrt{p}$) the quantity $\kappa_s^{(p)}$ should scale like $p^{s/2}$, and further, being bold, one might even guess that this limiting function will be the step function with a jump at $t=2$.
If this were the case, then we would have the asymptotic behaviour
\[
\kappa_s^{(p)} = p^{s/2}\int_0^\infty t^{s-1}\left( 1-F(t\sqrt{p} \right)dt
\sim p^{s/2}\int_0^2t^{s-1}dt=(4p)^{s/2}
\]
as $p\to \infty$.

In a recent work, Forrester et al.~\cite{MR2747559} gave arguments that $F_p(t)$ indeed exhibits this behavior.
Moreover, the authors succeeded to show that the fluctuations around the mean value are typically of order $p^{-1/6}$ and follow the famous Tracy-Widom probability law.
However, from a mathematical point of view, the derivation in \cite{MR2747559} has to be considered highly non-rigorous, as the arguments involve certain saddle point heuristics which are based on some unproved assumptions.
Very recently, Liechty~\cite{Liechty} studied discrete Gaussian orthogonal polynomials and, based on a Riemann-Hilbert approach, derived asymptotic results that can be used to give a complete rigorous proof of the results in \cite{MR2747559}.
%
%
%
%
\end{remark}

\section{Some multidimensional Dirichlet series}
\label{sec:jacobi}
In this section we study the multidimensional Dirichlet series
\[
	Z_{\vec a}(z)
	=\sum_{\vec m\neq\vec 0}\frac{m_0^{a_0}\dots m_{p-1}^{a_{p-1}}}{(m_0^2+\dots+m_{p-1}^2)^z}
	=\sum_{\vec m\neq \vec 0}\frac{\vec m^\vec a}{|\vec m|_2^{2z}},
\]
where $\vec m=(m_0,\dots,m_{p-1})$ ranges over $\Z^p\setminus\left\{ 0 \right\}$,
for $\vec a=(a_0,\dots,a_{p-1})\in\Z^p$, $\vec a\ge \vec 0$. 
Our goal is to establish the analytic continuation of $Z_{\vec a}(z)$ to a meromorphic function
and the determination of its poles. Also, we need information on the growth of $Z_{\vec a}(z)$
as $|z|\to\infty$ in some vertical strip.

It follows from the definition that $Z_{a_0,\dots,a_{p-1}}(z)=Z_{a_{\sigma(0)},\dots,a_{\sigma(p-1)}}(z)$
for every permutation $\sigma\in S_p$.
If $p=1$ then
\[ Z_a(z)=2\left[ \text{$a$ even} \right]\zeta(2z-a), \]
where $[\textrm{Statement}]$ is Iverson's notation, that is
\[
	[\textrm{Statement}] = \begin{cases} 1 & \textrm{if 'Statement' is true,} \\ 0 & \textrm{otherwise.} \end{cases}
\]
If $a_{p-1}$ is odd, the definition shows that $Z_{a_0,\dots,a_{p-2},a_{p-1}}(z)=0$.
Consequently, we may assume that the parameters $a_0,\dots,a_{p-1}$ are even.

The analytic continuation of $Z_{2\vec a}(z)$ is accomplished very much in the spirit of one of Riemann's methods for $\zeta(z)$
(see, e.g., \cite[Section 2.6]{MR882550}).
In fact we have
\begin{align}
	\frac{(2\pi i)^{2|\vec a|_1}}{\pi^z}Z_{2\vec a}(z)\Gamma(z)
	&=
	\int_0^\infty t^{z-1}\left(\left( \prod_{j=0}^{p-1} \vartheta_{2a_j}(t) \right)-\left[\vec a=\vec 0 \right]\right)dt,
	\label{eq:Z_riemann}
\end{align}
where $\vartheta_a(t)=\theta_a(0,it)$ and where
\[
\theta_a(x,y)=\frac{\partial^a}{\partial x^a}\theta(x,y)=\sum_{n=-\infty}^\infty (2\pi in)^ae^{2\pi i(xn+n^2y/2)},
\qquad \Im(y)>0,
\]
is the $a$-th derivative with respect to $x$ of $\theta(x,y)=\sum_n e^{2\pi i(xn+n^2y/2)}$, a variant of one of Jacobi's theta functions.
Here, Equation~(\ref{eq:Z_riemann}) is obtained by substitution of Euler's integral for the gamma function,
viz. $\Gamma(z)=\int_0^\infty t^{z-1}e^{-t}dt$, and
the series definition for $Z_{2\vec a}$ on the left hand side of the equation above followed by
interchanging summation and integration as well as a change of variables in the integral.

We are now going to extract information on the poles of $Z_{2\vec a}(z)$ from the integral~(\ref{eq:Z_riemann}).
This task is accomplished with the help of a generalised reciprocity relation (see Corollary~\ref{cor:vartheta_a_identity}),
which is a consequence of the following two results, stated in Lemma~\ref{lem:theta_derivatives} and
Proposition~\ref{lem:theta_derivative_reciprocity}. 
This relation generalises Jacobi's reciprocity law for $\theta(x,y)$, and is proved following along the lines of the proof
of the reciprocity law in \cite[Section 2.3]{MR1889901}.
\begin{lemma}\label{lem:theta_derivatives}
Let $(f_a(x,y))_{a\ge 0}$ be a sequence of functions which are entire with respect to $x$ for every fixed $y$ with $\Im(y)>0$.
If $(f_a(x,y))_{a\ge 0}$ satisfies the conditions
\renewcommand{\theenumi}{\roman{enumi}}
\begin{enumerate}
	\item \label{lem:enum:theta_derivatives_p1}
		$f_a(x+1,y) = f_a(x,y)$ 
	\item\label{lem:enum:theta_derivatives_p2}
	 	$f_a(x-y,y) = e^{2\pi i(x-y/2)}\sum\limits_{k=0}^a\binom{a}{k}f_k(x,y)$
\end{enumerate}
then we have
\begin{equation}
	\label{eq:theta_derivatives}
	f_a(x,y)=\sum_{k=0}^a\binom{a}{k}\frac{c_0^{(k)}(y)}{(2\pi i)^{a-k}}\theta_{a-k}(x,y),
\end{equation}
where
\[
c_0^{(k)}(y)=\int_0^1 f_k(x,y)dx
\]
is the constant term in the Fourier expansion of $f_k(x,y)$ as a function in $x$.
\end{lemma}
\begin{proof}
		Condition~(\ref{lem:enum:theta_derivatives_p1}) implies the convergent Fourier expansion
		($f_a(x,y)$ being understood as a function of $x$)
\[
	f_a(x,y)=\sum_n c_n^{(a)}(y)e^{2\pi i(xn+n^2y/2)}
\]
for $a\ge 0$ which shows that
\[
	e^{-2\pi i(x-y/2)}f_a(x-y,y)=\sum_nc_{n+1}^{(a)}(y)e^{2\pi i(xn+n^2y/2)}.
\]
Now, this last equation and Condition~(\ref{lem:enum:theta_derivatives_p2}) together imply the recursion
\[
	c_{n+1}^{(a)}(y)=\sum_{k=0}^a\binom{a}{k}c_n^{(k)}(y),
\]
which yields
\[
	c_n^{(a)}(y)=\sum_{k=0}^a\binom{a}{k}n^{a-k}c_0^{(k)}(y).
\]
This proves the lemma.
\end{proof}

\begin{proposition}
	We have
	\begin{align}
		\sum_{k=0}^{\lfloor\frac{a}{2}\rfloor}\binom{a}{2k}\frac{(2k)!}{k!}\pi^{k}\left(\frac{y}{i}\right)^{a-k+1/2}\theta_{a-2k}(x,y)
		&= \label{eq:theta_derivative_reciprocity}
		e^{-i\pi x^2/y}\sum_{k=0}^a\binom{a}{k}(-x)^ki^{k-a}(2\pi)^k\theta_{a-k}\left(\frac{x}{y},-\frac{1}{y}\right). 
	\end{align}
	\label{lem:theta_derivative_reciprocity}
\end{proposition}
\begin{proof}
	We prove the claim by applying Lemma~\ref{lem:theta_derivatives} to the functions
	\[
		f_a(x,y)=\sum_n\left( -\frac{x+n}{y} \right)^ae^{-i\pi (x+n)^2/y},\qquad a\ge 0.
	\]
	Condition~(\ref{lem:enum:theta_derivatives_p1}) of Lemma~\ref{lem:theta_derivatives} is clearly satisfied by $f_a(x,y)$.
	For Condition~(\ref{lem:enum:theta_derivatives_p2}) we calculate
	\begin{equation*}
			f_a(x-y,y) = \sum_n\left( 1-\frac{x+n}{y} \right)^ae^{-i\pi (x+n-y)^2/y}
			= e^{2\pi i(x-y/2)}\sum_{k=0}^a\binom{a}{k}\sum_n\left( -\frac{x+n}{y} \right)^ke^{-i\pi(x+n)^2/y}.
 	\end{equation*}
	It remains to determine the coefficients $c_0^{(a)}(y)$ of Lemma~\ref{lem:theta_derivatives}. Short calculations show that
	\[
	c_0^{(a)}(y)=\int_0^1f_a(x,y)dx=2\left[ \textrm{$a$ even} \right]\int_0^\infty\left( \frac{x}{y} \right)^ae^{-i\pi x^2/y}dx.
	\]
	In particular we have for $a=0$
	\[
		c_0^{(0)}(y)=\int_{-\infty}^\infty e^{-i\pi x^2/y}dx=\sqrt \frac{y}{i}.
	\]
	Note that the evaluation of the integral above is true for $y=it$ for some $t>0$ and analytic continuation then proves the
	correctness for general $y$ with $\Im(y) >0$.
	If $a>0$ then integration by parts yields the recursion
	\begin{equation*}
			c_0^{(2a)}(y)=
			2\int_0^\infty\left( \frac{x}{y} \right)^{2a}e^{-i\pi x^2/y}dx=
			\frac{2a-1}{i\pi y}\int_0^\infty\left( \frac{x}{y} \right)^{2a-2}e^{-i\pi x^2/y}dx
			= \frac{2a-1}{2i\pi y}c_0^{(2a-2)}(y),
	\end{equation*}
	and we obtain
	\[
		c_0^{(a)}(y)=
		\begin{cases}
				0 & \text{if $a$ is odd} \\
				\frac{a!}{(4\pi i y)^{a/2}(a/2)!}\sqrt\frac{y}{i} & \text{if $a$ is even.}
		\end{cases}
	\]
	Hence, by Lemma~\ref{lem:theta_derivatives} we have
	\[
	f_a(x,y)=\sum_{k=0}^{\lfloor\frac{a}{2}\rfloor}\binom{a}{2k}\frac{(2k)!}{k!}\frac{\pi^k}{(2\pi i)^a}\left( \frac{y}{i} \right)^{-k+1/2}\theta_{a-2k}(x,y).
	\]
	On the other hand, expanding the binomial term shows that
	\begin{align*}
		f_a(x,y)&= (-y)^{-a}e^{-i\pi x^2/y}\sum_{k=0}^a\binom{a}{k}x^k\sum_n n^{a-k}e^{-2\pi i(xn+n^2/y)}
		\\
		&=y^{-a}e^{-i\pi x^2/y}\sum_{k=0}^a\binom{a}{k}(-x)^k(2\pi i)^{k-a}\theta_{a-k}\left( \frac{x}{y},-\frac{1}{y} \right).
	\end{align*}
	The last two representations for $f_a(x,y)$ prove the lemma.
\end{proof}

Putting $a=0$ in Equation~(\ref{eq:theta_derivative_reciprocity}), we obtain the reciprocity law
for Jacobi's theta functions in the form
\begin{align*}
		\sqrt{\frac{y}{i}}\theta(x,y) &= e^{-i\pi x^2/y}\theta\left( \frac{x}{y},-\frac{1}{y} \right).
\end{align*}
\begin{corollary}
	The functions $\vartheta_a(y)=\theta_a(0,iy)$, $a\ge 0$, satisfy the relation
	\begin{equation}
		\vartheta_a(y)
		=
		i^a\sum_{k=0}^{\lfloor\frac{a}{2}\rfloor}\binom{a}{2k}\frac{(2k)!}{k!}\pi^{k}\left( \frac{1}{y} \right)^{a-k+1/2}\vartheta_{a-2k}\left( \frac{1}{y} \right),
		\qquad y>0.
		\label{eq:vartheta_a_identity}
	\end{equation}
	\label{cor:vartheta_a_identity}
\end{corollary}
\begin{proof}
The corollary follows from Equation~(\ref{eq:theta_derivative_reciprocity}) upon setting $x=0$ and replacing
$y$ by $i/y$.
\end{proof}

We can now prove the main result of this section.
\begin{lemma}
		\label{lem:Z_continuation}
		The function $Z_{2\vec a}(z)$ can be analytically continued to a meromorphic function having a single pole of 
		order $1$ at $z=\frac{p}{2}+|\vec a|_1$ with residue
		\begin{align}
			\Res\limits_{z=\frac{p}{2}+|\vec a|_1} Z_{2\vec a}(z)
			&= 
			\frac{\pi^{p/2}}{\Gamma\left( \frac{p}{2}+|\vec a|_1 \right)}
			\left( \prod_{i=0}^{p-1}\frac{(2a_i)!}{4^{a_i}a_i!} \right).
			\label{eq:Z_continuation}
		\end{align}
		Furthermore, we have the representation
		\begin{multline*}
				Z_{2\vec a}(z)
				=
				-\frac{[\vec a=\vec 0]\pi^{z}}{\Gamma(z+1)}
				+\frac{\pi^{z-|\vec a|_1}}{\Gamma(z)} 
				 \frac{\prod_{i=0}^{p-1}\frac{(2a_i)!}{4^{a_i}a_i!}}{z-\frac{p}{2}-|\vec a|_1}
				+\frac{\pi^{z-2|\vec a|_1}}{(-4)^{|\vec a|_1}\Gamma(z)}
				 \int_1^\infty t^{z-1}\left(\left( \prod_{j=0}^{p-1} \vartheta_{2a_j}(t)\right)
				 	-\left[ \vec a=\vec 0 \right]\right)dt
				\\ +\frac{\pi^{z-2|\vec a|_1}}{(-4)^{|\vec a|_1}\Gamma(z)}\int_1^\infty t^{-z-1} \left( \left(\prod_{j=0}^{p-1} \vartheta_{2a_j}\left(\frac{1}{t}\right)\right)
				-(-\pi)^{|\vec a|_1}\left( \prod_{i=0}^{p-1}\frac{(2a_i)!}{a_i!} \right)t^{p/2+|\vec a|_1} \right)dt,
		\end{multline*}
		where the two integrals above define entire functions with respect to $z$.
		For any non-negative integer $k$ we have
		\[
		Z_{2\vec a}(-k) = \begin{cases} -1 & \textrm{if $\vec a=\vec 0$ and $k=0$,} \\ 0 & \textrm{otherwise.}\end{cases}
		\]
\end{lemma}
\begin{proof}
	Consider again Equation~(\ref{eq:Z_riemann}), viz.
	\begin{align*}
		\frac{(2\pi i)^{2|\vec a|_1}}{\pi^z}Z_{2\vec a}(z)\Gamma(z)
		&=
		\int_0^\infty t^{z-1}\left(\left( \prod_{j=0}^{p-1} \vartheta_{2a_j}(t) \right)-\left[\vec a=\vec 0 \right]\right)dt.
	\end{align*}
	We split the integral above into two parts, one over $[0,1]$ and one over $[1,\infty)$. The second integral
	is seen to define an entire function with respect to $z$. We consider the first integral
	\[
	\int_0^1 t^{z-1}\left(\left( \prod_{j=0}^{p-1} \vartheta_{2a_j}(t) \right)-\left[\vec a=\vec 0 \right] \right)dt
	=
	-\frac{[\vec a=\vec 0]}{z}+\int_0^1 t^{z-1}\left( \prod_{j=0}^{p-1} \vartheta_{2a_j}(t) \right)dt.
	\]
	By virtue of (\ref{eq:vartheta_a_identity}) we obtain
	\begin{multline*}
		\left( \prod_{j=0}^{p-1} \vartheta_{2a_j}(t) \right)
		-(-\pi)^{|\vec a|_1}\left( \prod_{j=0}^{p-1}\frac{(2a_j)!}{a_j!} \right)t^{-p/2-|\vec a|_1}
		\\ =(-1)^{|\vec a|_1}\left( \sum_{\substack{\vec 0\le\vec k\le \vec a \\ \vec k\neq \vec a}}
		\left( \prod_{j=0}^{p-1}
						\binom{2a_j}{2k_j}
						\frac{(2k_j)!}{k_j!}
						\pi^{k_j}
						t^{k_j-2a_j-1/2}
						\vartheta_{2a_j-2k_j}\left( \frac{1}{t}
					\right)
			\right)
		\right)
		\\ +(-\pi)^{|\vec a|_1}\left( \prod_{j=0}^{p-1}\frac{(2a_j)!}{a_j!} \right)t^{-p/2-|\vec a|_1}
		\left( \vartheta\left( \frac{1}{t} \right)^p-1  \right).
	\end{multline*}

	Now, since for $\vec a\neq\vec 0$ the integrals
	\[
	\int_0^1t^{z-1}\left( \vartheta\left( \frac{1}{t} \right)^p-1  \right)dt
	\qquad\textrm{and}\qquad
	\int_0^1t^{z-1}\left( \prod_{j=1}^p\vartheta_{2a_j}\left(\frac{1}{t}\right) \right)dt
	\]
	define entire functions with respect to $z$ we see that
	\[
	\int_0^1 t^{z-1} \left( \left(\prod_{j=1}^p \vartheta_{2a_j}(t)\right)
	-(-\pi)^{|\vec a|_1}\left( \prod_{j=0}^{p-1}\frac{(2a_j)!}{a_j!} \right)t^{-p/2-|\vec a|_1}\right)dt
	\]
	defines an entire function with respect to  $z$, too.

	Combining all the parts and noting that
	\[
		(-\pi)^{|\vec a|_1}\left( \prod_{j=0}^{p-1}\frac{(2a_j)!}{a_j!} \right)\int_0^1 t^{z-1-p/2-|\vec a|_1}dt	
		=\frac{(-\pi)^{|\vec a|_1}\prod_{j=0}^{p-1}\frac{(2a_j)!}{a_j!}}{z-\frac{p}{2}-|\vec a|_1}
	\]
	we obtain the representation for $Z_{2\vec a}$ claimed in the lemma.
	The evaluations at the non-positive integers immediately follow from this representation.
\end{proof}

We close this section with a result on the growth of $Z_{2\vec a}(\sigma+it)$ as $|t|\to\infty$.
\begin{lemma}
		For $\sigma\in\R$ fixed we have the estimate
		\begin{equation}
				Z_{2\vec a}(\sigma+it)=O\left( e^{\varepsilon|t|} \right),\qquad |t|\to\infty,
				\label{eq:Z_vertical_growth}
		\end{equation}
		for any $\varepsilon>0$.
		\label{lem:Z_vertical_growth}
\end{lemma}
\begin{proof}
		Mellin transform asymptotics show that
		\begin{align*}
 				\vartheta_{2a}(t) &= \frac{(-4\pi^2)^a}{\pi^{a+1/2}}\frac{\Gamma(a+1/2)}{t^{a+1/2}}+O\left( t^M \right), \qquad t\to 0, \\
 				\vartheta_{2a}(t) &= [a=0]+O\left( t^{-M} \right),\qquad t\to\infty,
 		\end{align*}
 		for any $M>0$. Consequently, we have for $\vec a\in\N^p$ the asymptotics
		\begin{align*}
				\left( \prod_{i=0}^{p-1}\vartheta_{2a_i}(t) \right)-[\vec a=\vec 0] &= 
				\frac{(-4\pi)^{|\vec a|_1}}{t^{|\vec a|_1+p/2}}\frac{\prod_{j=0}^{p-1}\Gamma(a_j+1/2)}{\pi^{p/2}}
				-[\vec a=\vec 0]+O\left( t^M \right), \qquad t\to 0, \\
				\left( \prod_{i=0}^{p-1}\vartheta_{2a_i}(t) \right)-[\vec a=\vec 0] &= 
				O\left( t^{-M} \right),\qquad t\to\infty,
		\end{align*}
		for any $M>0$.
		Now, by \cite[Prop. 5]{MR1337752} we see that the Mellin transform of $\left( \prod_{i=0}^{p-1}\vartheta_{2\vec a}(t) \right)-[\vec a=\vec 0]$, viz.
		\begin{align*}
				f^\ast_{2\vec a}(z) &= \frac{(2\pi i)^{2|\vec a|_1}}{\pi^z}Z_{2\vec a}(z)\Gamma(z),
		\end{align*}
		satisfies
		\[
		f^\ast_{2\vec a}(\sigma+it)=O\left( e^{-(\pi/2-\varepsilon)|t|} \right),\qquad |t|\to\infty
		\]
		for any $\varepsilon>0$ and $\sigma$ in any closed subinterval of $(|\vec a|_1+p/2,\infty)$, which can be extended to any closed
		subinterval of $(-\infty,\infty)$ (see the proof of \cite[Prop. 4]{MR1337752} for details).
		The result now follows from the behaviour of the gamma function along vertical lines, viz.
		\[ \Gamma(\sigma+it) \sim \sqrt{2\pi}|t|^{\sigma-1/2}e^{-\pi|t|/2},\qquad |t|\to\infty. \]
\end{proof}

\section{Proof of Theorem~\ref{thm:sth-moment-asymptotics}: The moments of the height distribution}\label{sec:moments}
The goal of this section is to obtain an asymptotic expression for the $s$-th moment
$\E H_{n,p}^s$, where $\E$ denotes the expectation with respect to $\P$,
of this random variable as the length of the watermelons tends to infinity. Clearly, we have
\begin{align}
		\E H_{n,p}^s &= 		
	\frac{1}{M_{2n}^{(p)}}\sum_{h\ge 1}\left( h^s-(h-1)^s \right) \left(M_{2n}^{(p)}-M_{2n,h}^{(p)}\right),
	\qquad s\ge 1.
	\label{eq:wm_with_wall_sth_moment_def}
\end{align}

The proof of Theorem~\ref{thm:sth-moment-asymptotics} is based on a series of lemmas, and can be roughly summarised as follows.
For determining the asymptotics of $\E H_{n,p}^s$ 
we proceed as follows.
First, we find expressions in terms of determinants for the quantities $M_{2n,h}^{(p)}$ and $M_{2n}^{(p)}$.
This is accomplished by an application of a theorem by Lindstr\"om--Gessel--Viennot,
respectively of a theorem by Gessel and Zeilberger.
Second, we obtain asymptotics for
\begin{align} 
	&\label{eq:asymptotics_for}
	M_{2n,h}^{(p)}
	\qquad\textrm{and}\qquad
	\sum_{h\ge 1}\left( h^s-(h-1)^s \right)\left( M_{2n}^{(p)}-M_{2n,h}^{(p)} \right).
\end{align}
The proof of Theorem~\ref{thm:sth-moment-asymptotics} itself can be found at the end of this section.

We start with exact expressions for $M_{2n,h}^{(p)}$ and $M_{2n}^{(p)}$.
\begin{lemma}
		\label{lem:exact_expressions}
		We have
		\begin{align}
				M_{2n}^{(p)} &=\det_{0\le i,j<p}\left(\binom{2n}{n+i-j}-\binom{2n}{n-1-i-j}\right) \label{eq:exact_M_2n}\\
				\intertext{and}
				M_{2n,h}^{(p)} &= \det_{0\le i,j<p}\left( \sum_{m\in\Z}\left(\binom{2n}{n+m(h+1)+i-j}-\binom{2n}{n+m(h+1)-1-i-j}\right) \right),\qquad h\ge 0. \label{eq:exact_M_2n_h}
		\end{align}
\end{lemma}
\begin{proof}[Proof (Sketch)]
For $h\ge 2p$ both equations follow from a theorem by Lindstr\"om--Gessel--Viennot
(see \cite[Corollary 3]{gesselviennot} or \cite[Lemma 1]{MR0335313}), respectively from a theorem of Gessel and Zeilberger~\cite{MR1092920}.
To be more specific, Equation~(\ref{eq:exact_M_2n}) follows from the type $C_p$ case of the main theorem in \cite{MR1092920}, while
Equation~(\ref{eq:exact_M_2n_h}) follows from the type $\tilde C_p$ case.

The reader should observe that the entries of the determinant~(\ref{eq:exact_M_2n}) are the numbers of lattice paths
from $(0,2i)$ to $(2n,2j)$ that do not cross the $x$-axis.
On the other hand, the entries of the determinant~(\ref{eq:exact_M_2n_h}) are the numbers of lattice paths from $(0,2i)$ to $(2n,2j)$
that do not cross the $x$-axis and have height smaller than $h$. These sums are obtained by a repeated reflection principle
(see, e.g., Mohanty~\cite[p.6]{MR554084}).
		
For $0\le h<2p$ 
the identity 
		 \begin{multline*}
		 	\sum_{m\in\Z}\left(\binom{2n}{n+m(h+1)+i-j}-\binom{2n}{n+m(h+1)-1-i-j}\right) \\
		 	=-\sum_{m\in\Z}\left(\binom{2n}{n+m(h+1)+(h-i)-j}-\binom{2n}{n+m(h+1)-1-(h-i)-j}\right)
		\end{multline*}
shows that the right hand side of (\ref{eq:exact_M_2n_h}) is equal to zero, since for $h=2i$ the $i$-th row of the determinant is equal to zero, and for $h=2i+1$ we see that the $i$-th and $(i+1)$-th row of the determinant
only differ by sign and thus are linear dependent.
\end{proof}
We now turn towards the problem of determining asymptotics for the expressions~(\ref{eq:asymptotics_for}).
Asymptotics for the total number of watermelons are easily established since the determinant in
(\ref{eq:exact_M_2n}) admits a simple closed form. The result is stated in the following lemma.
\begin{lemma}
		We have
 	\begin{align*}
		M_{2n}^{(p)} &= 
			4^{\binom{p}{2}}\left( \prod_{i=0}^{p-1}(2i+1)! \right)\binom{2n}{n}^pn^{-p^2}\left( 1+O\left( n^{-1} \right) \right)
	\end{align*}
	as $n\to\infty$.
	\label{lem:wm_with_wall_asymptotic_total}
\end{lemma}
\begin{proof}
		The determinant~(\ref{eq:exact_M_2n}) can be evaluated in closed form, e.g., by means of
		\cite[Theorem 30]{MR1701596}, and is in fact given by
		\[
		M_{2n}^{(p)} = \prod_{j=0}^{p-1}\frac{\binom{2n+2j}{n}}{\binom{n+2j+1}{n}}
		=\binom{2n}{n}^p\left( \prod_{j=0}^{p-1}(2j+1)! \right)
		\left( \prod_{j=0}^{p-1}\frac{(2n+2j)\dots(2n+1)}{(n+2j+1)(n+2j)^2\dots(n+1)^2} \right).
		\]
		This proves, upon determining asymptotics for the right-most product, the result as stated in the lemma.

		For a comprehensive discussion and references of this counting problem we refer to \cite[Section 4]{MR1801472}.
\end{proof}

Asymptotics for the second part of (\ref{eq:asymptotics_for}) are much harder to obtain. As a first step 
we note that
\begin{multline}
		\label{eq:difficult_part}
		\sum_{h\ge 1}\left( h^s-(h-1)^s \right)\left( M_{2n}^{(p)}-M_{2n,h}^{(p)} \right)\\
		=-\sum_{h\ge 1}\left( h^s-(h-1)^s \right)\sum_{\vec m\neq \vec 0}\det_{0\le i,j<p}\left( \binom{2n}{n+m_i(h+1)+i-j}-\binom{2n}{n+m_i(h+1)-1-i-j} \right)
\end{multline}
by (\ref{eq:exact_M_2n}) and (\ref{eq:exact_M_2n_h}), where the inner sum ranges over $\Z^p\setminus\{\vec 0\}$.

For determining asymptotics for (\ref{eq:difficult_part}) we closely follow the proof of
de Bruijn, Knuth and Rice~\cite{MR0505710} (in our case we have to overcome some additional difficulties).
For sake of convenience, we give a short plan of the proof. As a first step we factor $\binom{2n}{n}$ out of each
row of the determinant on the right-hand side of (\ref{eq:exact_M_2n_h}).
We then replace the quotients of binomial coefficients by its (sufficiently accurate) asymptotic series expansion, which is
determined with the help of Stirling's asymptotic series for the factorials (see Lemma~\ref{lem:approx_binom_quot}).
This shows that the asymptotic series expansion for (\ref{eq:difficult_part}) can be expressed in terms
of products of derivatives of Jacobi's theta functions we considered in the last section.
With the help of the Mellin transforms and the results of the last section we are able to derive asymptotics
for these functions (see Lemma~\ref{lem:g_a_asymptotics}).
In Lemma~\ref{lem:wm_with_wall_asymptotic_for_sum} we finally obtain the desired asymptotics for (\ref{eq:difficult_part}).

We start with the asymptotic series expansion for the quotients of binomial coefficients mentioned above.
\begin{lemma}
For $|m-z|\le n^{5/8}$ and $N>1$ we have the asymptotic expansion
		\begin{multline} \label{eq:approx_binom_quot}
				\frac{\binom{2n}{n+m-z}}{\binom{2n}{n}}
				= e^{-m^2/n}\\ \times
				\left(
				\sum_{u=0}^{4N+1}\left( -\frac{z}{\sqrt n} \right)^{u}
				\left(
					\phi_u\left( \frac{m}{\sqrt n} \right)
					+\sum_{l=1}^{3N+1}n^{-l}\sum_{k=0}^{u-1}
					\sum_{r=1}^{2l}F_{r,l}\binom{2r}{u-k}\phi_k\left( \frac{m}{\sqrt n}  \right)
					\left( \frac{m}{\sqrt n}  \right)^{2r+k-u}
				\right)
				+O\left( n^{-1-2N} \right)
				\right)
		\end{multline}
		as $n\to\infty$. 
		Here, the $F_{r,l}$ are some constants the explicit form of which is of no importance in the sequel, and 
		$(-1)^kk!\phi_k(w)$ is the $k$-th Hermite polynomial, that is
		\begin{align}
				\phi_k(z) &= \sum_{m\ge 0}\frac{(-1)^m}{m!}\binom{m}{k-m}(2z)^{2m-k},\qquad k\ge 0.
				\label{eq:phi_s_def}
 		\end{align}
		\label{lem:approx_binom_quot}
\end{lemma}
\begin{proof}
		For sake of convenience, set $x=(m-z)/n$. 
		With the help of Stirling's asymptotic series for the factorials we see that for $x$ sufficiently small,
		$|x|<\frac{1}{2}$, say, we have 
		\begin{multline*}
		\log\frac{\binom{2n}{n+m-z}}{\binom{2n}{n}}
		=\left( n+\frac{1}{2} \right)\log\frac{1}{1-x^2}-nx\log\frac{1+x}{1-x} \\
		+\sum_{k=1}^N\frac{B_{2k}n^{1-2k}}{2k(2k-1)}\left( 2-(1+x)^{1-2k}-(1-x)^{1-2k} \right)
		+O\left( n^{-1-2N} \right)
		\end{multline*}
		for all fixed $N>0$ as $n\to\infty$.
		Here, $B_{k}$ denotes the $k$-th Bernoulli number defined via $\sum_{k\ge 0}B_kt^k/k!=t/(e^t-1)$.
		
		For the range $|x|\le n^{-1/4}$ we further obtain by Taylor series expansion and some simplifications
		the expression
		\begin{multline*}
		\log\frac{\binom{2n}{n+m-z}}{\binom{2n}{n}}
		=-n\left( \sum_{r=1}^{4N+3}\frac{x^{2r}}{r(2r-1)} \right)
		+\frac{1}{2}\left( \sum_{r=1}^{4N+1}\frac{x^{2r}}{r} \right) \\
		-\left( \sum_{r=1}^{4N-1}\left( \sum_{k=1}^N\frac{B_{2k}n^{1-2k}}{k(2k-1)}\binom{-2k+1}{r} \right)x^{2r} \right)
		+O\left( n^{-1-2N} \right).
		\end{multline*}
		
		Further restricting ourselves to the range $|x|\le n^{-3/8}$
		we obtain, upon taking the exponential of both sides of the expression above and another Taylor series expansion,
		the asymptotic series expansion
		\[
		\frac{\binom{2n}{n+m-z}}{\binom{2n}{n}}
		=e^{-nx^2}
		\left(1+\sum_{r=1}^{4N+3}
		\left( \sum_{l=-\floor{r/2}}^{2N-\floor{3r/4}}F_{r,l+r}n^{-l} \right)x^{2r}+O\left(n^{-1-2N}\right)  \right)
		\]
		for some constants $F_{r,l}$.

		Now, if $N>1$ we obtain upon interchanging the two sums on the right-hand side above, replacing $x$ with
		its defining expression $(m-z)/n$ and simple rearrangements the expression
		\[
		\frac{\binom{2n}{n+m-z}}{\binom{2n}{n}}
		=e^{-(m-z)^2/n}
		\left( 1+\sum_{l=1}^{3N+1}n^{-l}
		\sum_{r=1}^{2l}F_{r,l}\left( \frac{m-z}{\sqrt n} \right)^{2r}+O\left( n^{-1-2N} \right)\right)
		\]
		Finally, expanding $e^{-(m-z)^2/n}$ in the expression above in the form
		\[
		e^{-(m-z)^2/n} =
		e^{-m^2/n}\sum_{k\ge 0}\phi_k\left( \frac{m}{\sqrt n} \right)\left( -\frac{z}{\sqrt n} \right)^k,
		\]
		and collecting powers of $z$, we obtain the result.
		Here, the $\phi_k(m/\sqrt n)$ represent certain polynomials the explicit form of which is given in the lemma.
\end{proof}

We mentioned before, that the non-normalised $s$-th moment (\ref{eq:difficult_part}) is a linear combination of certain
functions related to products of the functions $\vartheta_{2a}(t)$, $a\ge 0$,
considered in the last section. In the next lemma
we obtain asymptotics for these functions with the help of the Mellin transform and the results proved in the last section.
\begin{lemma} For $\vec a\in\Z^p$, $\vec a\ge\vec 0$, and $k\in\N$ define the function
		\begin{align}
				g_{k,\vec a}(n) &= \sum_{h\ge 1}(h+1)^k
				\sum_{\substack{\vec m\in (h+1)\Z^p \\ \vec m\neq \vec 0}}e^{-|\vec m|_2^2/n}\left( \frac{\vec m}{\sqrt n} \right)^{2\vec a}.
				\label{eq:g_a_def}
		\end{align}
		For any fixed $M>0$ we have the asymptotics
		\begin{equation}
			\label{eq:g_a_asymptotics}
			g_{k,\vec a}(n)=
			\left( \prod_{j=0}^{p-1}\frac{(2a_j)!}{4^{a_j}a_j!} \right)\Omega_k(n)
			+\omega_{k,\vec a}n^{(k+1)/2}
			+\left(1-B_{k+1}\frac{(-1)^k}{k+1} \right)\left[ \textrm{$\vec a=\vec 0$}\right] +O\left( n^{-M} \right)
		\end{equation}
		as $n\to\infty$, where
		\[
		\Omega_k(n) = (n\pi)^{p/2}\times
		\begin{cases} \gamma-1+\log\sqrt n & \textrm{if $p=k+1$,} \\
					  \zeta(p-k)-1 & \textrm{else},
		\end{cases}
		\]
		and
		\[
		\omega_{k,\vec a}= \frac{1}{2}\times
		\begin{cases} 
				\lim\limits_{z\to p/2}
						\left( Z_{2\vec a}(z+|\vec a|_1)\Gamma(z+|\vec a|_1)
						-\left( \prod\limits_{j=0}^{p-1}\frac{(2a_j)!}{4^{a_j}a_j!} \right)\frac{\pi^{p/2}}{z-\frac{p}{2}}
						\right)
								& \textrm{if $p=k+1$,} \\
								\Gamma\left(\frac{k+1}{2}+|\vec a|_1  \right)Z_{2\vec a}\left( \frac{k+1}{2}+|\vec a|_1 \right)
					  			& \textrm{else.}
		\end{cases}
		\]
		Here,
		$\gamma=0.5772\dots$ is the Euler-Mascheroni constant.
		\label{lem:g_a_asymptotics}
\end{lemma}
\begin{proof}
		First, note that the function $g_{k,\vec a}(n)$ can be written in terms of derivatives of theta functions, viz.
		\begin{align*}
				g_{k,\vec a}(n) &=
				(-4\pi)^{-|\vec a|_1}\sum_{h\ge 1}(h+1)^k\left( \frac{(h+1)^2}{n\pi} \right)^{|\vec a|_1}
				\left( \prod_{j=0}^{p-1}\vartheta_{2a_j}\left( \frac{(h+1)^2}{n\pi} \right)-[\vec a=\vec 0] \right).
		\end{align*}
		Now, by the harmonic sum rule and Equation~(\ref{eq:Z_riemann}),
		the Mellin transform of $g_{k,\vec a}(n)$ is seen to be 
		\begin{align*}
				g^\ast_{k,\vec a}(z) &= \int_0^\infty g_{k,\vec a}(x^{-1})x^{z-1}dx
			=\left( \zeta(2z-k)-1 \right) \Gamma(z+|\vec a|_1)Z_{2\vec a}(z+|\vec a|_1),
			\qquad \Re(z)>\frac{1}{2}\max\left\{ p,k+1 \right\}.
	\end{align*}
		Consequently, the function $g_{k,\vec a}(n)$ can be represented with the help of the inverse Mellin transform by the contour integral
		\begin{align*}
				g_{k,\vec a}(n)=\frac{1}{2\pi i}\int_{c-i\infty}^{c+i\infty}g^{\ast}_{k,\vec a}(z)n^z dz,
				\qquad c>\frac{1}{2}\max\left\{ p,k+1 \right\}.
		\end{align*}
		Asymptotics are now being obtained by pushing the line of integration to the left and taking into account the
		residues of the poles of the integrand.

		From the well-known analytic behaviour of the gamma and the zeta function (see, e.g., \cite{MR0178117}) and the 
		analytic behaviour of $Z_{2\vec a}(z)$ as given by Lemma~\ref{lem:Z_continuation} we infer that
		the integrand $g_{k,\vec a}^{\ast}(z)n^z$ has potential poles at $z=p/2$, $z=(k+1)/2$ and $z=-|\vec a|_1-m$ for
		$m\in\N$.
		For $p\neq k+1$ all poles are of order one. Furthermore, the residues are given by
		\begin{align*}
				\Res_{z=p/2} g^{\ast}_{k,\vec a}(z)n^z
				&= \left(\zeta(p-k)-1\right)\left( \prod_{j=0}^{p-1}\frac{(2a_j)!}{4^{a_j}a_j!} \right)(n\pi)^{p/2}\\
				\Res_{z=(k+1)/2} g^{\ast}_{k,\vec a}(z)n^z
				&=\frac{1}{2}\Gamma\left(\frac{k+1}{2}+|\vec a|_1  \right)Z_{2\vec a}\left( \frac{k+1}{2}+|\vec a|_1 \right)n^{(k+1)/2} \\
				\Res_{z=-|\vec a|_1-m} g^{\ast}_{k,\vec a}(z)n^z
				&= -\left( B_{k+1}\frac{(-1)^k}{k+1}-1 \right)\left[ \textrm{$\vec a=\vec 0$ and $m=0$} \right],
		\end{align*}
		where $B_l$ denotes the $l$-th Bernoulli number defined via $\sum_{l\ge 0}B_l t^l/l!=t/( e^t-1)$.
		
		In the case $p=k+1$, the only difference is the pole at $z=p/2$, which is now a pole of order two.
		By Lemma~\ref{lem:Z_continuation} we know that
		\[
		\Res_{z=p/2} Z_{2\vec a}(z+|\vec a|_1)\Gamma(z+|\vec a|_1)
		=\left( \prod_{j=0}^{p-1}\frac{(2a_j)!}{4^{a_j}a_j!} \right)\pi^{p/2},
		\]
		and consequently, we have 
		\begin{multline*}
				\Res_{z=p/2} g^{\ast}_{k,\vec a}(z)n^z
				=
				\left( \prod_{j=0}^{p-1}\frac{(2a_j)!}{4^{a_j}a_j!} \right)
						\left( \gamma-1+\log\sqrt n\right)(\pi n)^{p/2} \\
						+\frac{n^{p/2}}{2}\lim_{z\to p/2}
						\left( Z_{2\vec a}(z+|\vec a|_1)\Gamma(z+|\vec a|_1)
						-\left( \prod_{j=0}^{p-1}\frac{(2a_j)!}{4^{a_j}a_j!} \right)\frac{\pi^{p/2}}{z-\frac{p}{2}}
						\right).
		\end{multline*}
		Note that the limit above is equal to the constant term in the Laurent expansion of
		$Z_{2\vec a}(z+|\vec a|_1)\Gamma(z+|\vec a|_1)$ around its pole $z=p/2$.

		For completing the proof we have to show the admissibility of the displacement of the contour of integration above.
		But this follows by well known estimates for the gamma and the zeta function along vertical lines in the
		complex plane together with Lemma~\ref{lem:Z_vertical_growth}.
		See \cite{MR1337752} for details.
\end{proof}
This last lemma finally enables us to determine the asymptotics for the
non normalised $s$-th moment~(\ref{eq:difficult_part}).
\begin{lemma} \label{lem:wm_with_wall_asymptotic_for_sum}
		We have the asymptotics
		\begin{multline*} 
			\sum_{h\ge 1}\left( h^s-(h-1)^s \right)\left( M_{2n}^{(p)}-M_{2n,h}^{(p)} \right) \\
			=2^{-p}\binom{2n}{n}^pn^{-p^2}
			\left(s\lambda_sn^{s/2}-3\binom{s}{2}\lambda_{s-1}n^{(s-1)/2}
			-\frac{3}{2}2^{p^2}\left( \prod_{i=0}^{p-1}(2i+1)! \right)
			+O\left( n^{s/2-1}+n^{p/2-p^2}\log n \right) \right)
		\end{multline*}
		as $n\to\infty$, where
		\[
		\lambda_k=-\sum_{\vec a\ge\vec 0}(-4)^{|\vec a|_1}
		\det_{0\le i,j<p}\left( \frac{(2i+2j+2)!}{(i+j+1-a_i)!\ (2a_i)!} \right)\omega_{k-1,\vec a},
		\qquad k>0,
		\]
		with $\omega_{k-1,\vec a}$ being defined in Lemma~\ref{lem:g_a_asymptotics}.
\end{lemma}
\begin{proof}
		Substituting the determinant expressions (\ref{eq:exact_M_2n}) and (\ref{eq:exact_M_2n_h}) for
		$M_{2n}^{(p)}$ and $M_{2n,h}^{(p)}$ we see that
		\begin{multline*}
				\sum_{h\ge 1}\left( h^s-(h-1)^s \right)\left( M_{2n}^{(p)}-M_{2n,h}^{(p)} \right) \\
			=\sum_{h\ge 1}\left( (h-1)^s-h^s \right)
			\sum_{\vec m\neq\vec 0}
			\det_{0\le i,j<p}\left( \binom{2n}{n+m_i(h+1)+i-j}-\binom{2n}{n+m_i(h+1)-1-i-j} \right).
		\end{multline*}
		Instead of determining asymptotics for the right-hand side expression above directly we consider the more
		general quantity
		\begin{align*}
			D_n(\vec x,\vec y,z)
			&=
			\sum_{h\ge 1}\left( (h-1)^s-h^s \right)
			\sum_{\substack{\vec m\in (h+1)\Z^p\\ \vec m\neq \vec 0}}
			\det_{0\le i,j<p}\left( \binom{2n}{n+m_i+x_i-y_j}-\binom{2n}{n+m_i-z-x_i-y_j} \right).
		\end{align*}
		Now, we factor $\binom{2n}{n}$ out of each row of the determinant above, and restrict the sum above to those
		$(p+1)$-tuples $(h,m_0,\dots,m_{p-1})$ such that for $i=0,\dots,p-1$ we have $|(h+1)m_i|\le n^{1/2+\varepsilon}$
		for some fixed $\varepsilon$ satisfying $0<\varepsilon\le 1/8$.
		Since, by Stirling's formula, we have
		\[ 
			\frac{\binom{2n}{n+\alpha}}{\binom{2n}{n}}=O\left( e^{-n^{2\varepsilon}} \right),
			\qquad n\to\infty,
		\]
		whenever $|\alpha|\ge n^{1/2+\varepsilon}$, we see that the sum of all terms failing to satisfy the condition
		above is $O\left( n^{-M} \right)$ for all $M>0$ and, therefore, is negligible.

		In the remaining sum we replace all quotients of binomial coefficients by their asymptotic series
		expansion as given by Lemma~\ref{lem:approx_binom_quot}.
		Having done so, we extend the range of summation to $\N\times(\Z^p-\left\{ \vec 0 \right\})$. This adds some
		additional terms, their sum being exponentially small and, therefore, again negligible.
		This technique of truncating the (exponentially small) tail of the exact sum,
		replacing the addends by their asymptotic expansion 
		and finally adding a new (exponentially small) tail to the resulting sum has also been applied by
		de Bruijn, Knuth and Rice~\cite{MR0505710}.
		
		This procedure yields, upon noticing some cancellations due to
		summation over $\vec m$ which eliminates all odd powers
		of $m_i$ for $i=0,\dots,p-1$, for arbitrary $N>0$ the expression
		\begin{multline}
				P_N(\vec x,\vec y,z) \\
				=
				\sum_{h\ge 1}\left( (h-1)^s-h^s \right)
				\sum_{\substack{\vec m\in (h+1)\Z^p \\ \vec m\neq \vec 0}}e^{-|\vec m|_2^2/n}
				\det_{0\le i,j<p}\left( \sum_{u=0}^{2N}\left( \frac{(y_j-x_i)^{2u}-(z+x_i+y_j)^{2u}}{n^{u}} \right)T_{2u;N}\left( \frac{m_i}{\sqrt n},n \right) \right),
				\label{eq:P_N_def}
		\end{multline}
		where 
		\begin{align}
				T_{u;N}(w,n) &=
				\phi_u(w)+\sum_{l=1}^{3N+1}n^{-l}\sum_{k=0}^{u-1}\sum_{r=1}^{2l}F_{r,l}\binom{2r}{u-k}\phi_k(w)w^{2r+k-u},
				\label{eq:T_def}
		\end{align}
		such that 
		\begin{equation}\label{eq:D_n_asymptotics}
				D_n(\vec x,\vec y,z) = 
				\binom{2n}{n}^p\Big( P_N(\vec x,\vec y,z) + O\left( n^{-2N-1}G_{s,\vec 0}(n) \right)\Big),
				\qquad n\to\infty.
		\end{equation}
		Here, the functions $\phi_k(w)$ are defined by (\ref{eq:phi_s_def}), and   
		\begin{align}
				\label{eq:G_def}
				G_{s,\vec a}(n) &= 
				\sum_{h\ge 1}\left( (h-1)^s-h^s \right)
				\sum_{\substack{\vec m\in(h+1)\Z^p \\ \vec m\neq \vec 0}}
				\left( \frac{\vec m}{\sqrt n} \right)^{2\vec a}e^{-|\vec m|_2^2/n}.
		\end{align}
		Clearly, $P_N(\vec x,\vec y,z)$ is a polynomial with respect to the variables
		$x_0,\dots,x_{p-1},y_0,\dots,y_{p-1},z$.
		Furthermore, expanding the determinants and interchanging summations in (\ref{eq:P_N_def}) reveals that
		$P_N(\vec x,\vec y,z)$ is of the form
		\begin{align}
				P_N(\vec x,\vec y,z) &=
				\sum_{\substack{\vec i,\vec j, \vec l \ge\vec 0 \\ |\vec i|_1+|\vec j|_1+|\vec l|_1\ \textrm{even}}}
				\frac{\vec x^{\vec i}\vec y^\vec j z^{|\vec l|_1}}{n^{(|\vec i|_1+|\vec j|_1+|\vec l|_1)/2}}
				\left(
				\sum_{\vec a\ge\vec 0}q_{\vec a,\vec i,\vec j,\vec l}\left( n^{-1} \right)
				G_{s,\vec a}(n)
				\right)
				\label{eq:P_N_expr1}
		\end{align}
		for some polynomials $q_{\vec a,\vec i,\vec j,\vec l}(n^{-1})$ in $n^{-1}$.
		Noting that $[w^{2a}]T_{2u+1;N}(w,n)=0$ for all $a$ by (\ref{eq:T_def}) we obtain upon
		extracting the corresponding coefficients in (\ref{eq:P_N_def}) the explicit representation
		\begin{align}
				\label{eq:q_a_explicit_rep}
				q_{\vec a,\vec v,\vec w,\vec l}(n^{-1})
				&=
				\det_{0\le i,j<p}\left( \binom{v_i+w_j+l_i}{v_i,w_j,l_i}\Big((-1)^{v_i}[l_i=0]-1 \Big)\left[ w^{2a_i} \right]T_{v_i+w_j+l_i;N}(w,n) \right),
		\end{align}
		and short calculations show that
		\begin{multline*}
				\left[ w^{2a} \right]T_{2u;N}(w,n) \\
				=
				\frac{(-1)^{u+a}}{(u+a)!}\binom{u+a}{2a}4^a
				+
				\sum_{l=1}^{3N+1}n^{-l}\sum_{r=1}^{2l}F_{r,l}\sum_{k=0}^{2u-1}\frac{(-1)^{u+a-r}}{(u+a-r)!}\binom{u+a-r}{2a+2u-2r-k}2^{2a+2u-2r-k}.
		\end{multline*}
		By expanding $(h-1)^s-h^s$ in powers of $(h+1)$ in (\ref{eq:G_def}) and interchanging summations we see that
		\begin{equation} \label{eq:G_a_expansion}
				G_{s,\vec a}(n) = \sum_{k=0}^{s-1}\binom{s}{k}(2^{s-k}-1)(-1)^{s-k}g_{k,\vec a}(n),
		\end{equation}
		where the functions $g_{k,\vec a}(n)$ are defined in Lemma~\ref{lem:g_a_asymptotics}.

		Thus, we are led to consider sums of the form
		\[ \sum_{\vec a\ge\vec 0}q_{\vec a,\vec v,\vec w,\vec l}(n^{-1})g_{k,\vec a}(n). \]
		Now, we replace $g_{k,\vec a}(n)$ by its asymptotic expansion~(\ref{eq:g_a_asymptotics}), viz.
		\[
		g_{k,\vec a}(n)=\left( \prod_{j=0}^{p-1}\frac{(2a_j)!}{4^{a_j}a_j!} \right)\Omega_k(n)
		+\omega_{k,\vec a}n^{(k+1)/2}
		+[\vec a=\vec 0]\left(1-B_{k+1}\frac{(-1)^k}{k+1}  \right) +O(n^{-M})
		\]
		as $n\to\infty$ for all $M>0$. The quantities $\Omega_k(n)$ and $\omega_{k,\vec a}$ have already been defined in
		Lemma~\ref{lem:g_a_asymptotics}.

		The multi-linearity of the determinant in (\ref{eq:q_a_explicit_rep}) then shows that
		\begin{multline*}
		\sum_{\vec a\ge\vec 0}q_{\vec a,\vec v,\vec w,\vec l}(n^{-1})
		\left( \prod_{j=0}^{p-1}\frac{(2a_j)!}{4^{a_j}a_j!} \right)\Omega_k(n) \\
		=
		\Omega_k(n)
		\det_{0\le i,j<p}\left( \binom{v_i+w_j+l_i}{v_i,w_j,l_i}\Big( [l_i=0](-1)^{v_i}-1 \Big)
		\sum_{a\ge 0}\frac{(2a)!}{4^aa!}\left[ w^{2a_i} \right]T_{v_i+w_j+l_i;N}(w,n)\right).
		\end{multline*}
		The sum inside the determinant is further seen to be
		\begin{multline*}
			\sum_{a\ge 0}\frac{(2a)!}{4^aa!}\left[ w^{2a} \right]T_{2u}(w,n) \\
			=
			\sum_{l=1}^{3N+1} n^{-l}\sum_{r=1}^{2l}F_{r,l}\sum_{k=2(u-r)}^{2u-1}\frac{(-1)^{a+u-r}}{2^{2r+k-2u}}
			\frac{(2r+k-2u)!}{(k+r-u)!}\sum_{a\ge 0}(-1)^a\binom{k+r-u}{a}\binom{2a}{2r+k-2u}.
		\end{multline*}
		By the Chu-Vandermonde summation formula we obtain for the innermost sum above
		\begin{multline*}
			\sum_{a\ge 0}(-1)^a\binom{k+r-u}{a}\binom{2a}{2r+k-2u} \\
			={}_2F_1
			\left[\begin{array}{c}
				-\lfloor k/2\rfloor, \frac{1}{2}+r-u+\lceil k/2\rceil \\
				\frac{1}{2}+\lceil k/2 \rceil - \lfloor k/2 \rfloor
			\end{array};1\right]
			=
			\frac{\Gamma\left( \frac{1}{2}+\lceil k/2 \rceil - \lfloor k/2 \rfloor \right)}
				 {\Gamma\left( \frac{1}{2}+\lceil k/2\rceil \right)}
			\frac{\Gamma(u-r)}{\Gamma(u-r- \lfloor k/2\rfloor )}
		\end{multline*}
		and, from the fact that $\lfloor k/2\rfloor \ge u-r$ on the right-hand side of the second to last equation above, 
		we conclude that all terms having $r<u$ vanish, since in these
		cases this last sum evaluates to zero.
		But this shows that
		\begin{align*}
			\sum_{a\ge 0}\frac{(2a)!}{4^aa!}\left[ w^{2a} \right]T_{2u}(w,n) 
			&=
			O\left( n^{-u/2} \right),
			\qquad n\to\infty,
		\end{align*}
		and we infer that
		\[
		\sum_{\vec a\ge\vec 0}q_{\vec a,\vec v,\vec w,\vec l}(n^{-1})
		\left( \prod_{j=0}^{p-1}\frac{(2a_j)!}{4^{a_j}a_j!} \right)\Omega_k(n)
		=
		O\left( \Omega_k(n)n^{-(|\vec v|_1+|\vec w|_1+|\vec l|_1)/2} \right),
		\qquad n\to\infty,
		\]
		and further, noting that $\Omega_k(n)=O(n^{p/2}\log n)$ as $n\to\infty$,
		\begin{multline}\label{eq:sum_q_g_asymptotics}
		\sum_{\vec a\ge\vec 0}q_{\vec a,\vec v,\vec w,\vec l}(n^{-1})g_{k,\vec a}(n)
		=
		\left( \sum_{\vec a\ge \vec 0}\omega_{k,\vec a} q_{\vec a,\vec v,\vec w,\vec l}(n^{-1})\right)n^{(k+1)/2} \\
		+q_{\vec 0,\vec v,\vec w,\vec l}(n^{-1})\left(1-B_{k+1}\frac{(-1)^k}{k+1}  \right)
		+O\left(n^{(p-|\vec v|_1-|\vec w|_1-|\vec l|_1)/2}\log n \right)
	\end{multline}
	as $n\to\infty$.

	Now, lets turn back to Equation~(\ref{eq:P_N_def}). 
	Since the determinants involved in the definition of $P_N(\vec x,\vec y,z)$ vanish
	whenever $x_i=x_j$ or $y_i=y_j$ for some $i\neq j$ or $x_i=-z-x_j$ or $y_i=-z-y_j$ for some $i$ and $j$,
	we conclude that $P_N(\vec x,\vec y,z)$, which is a polynomial with respect to the variables $\vec x$, $\vec y$ and
	$z$, is divisible by
	\[ 
		\left( \prod_{0\le i<j<p}(x_i-x_j)(y_i-y_j) \right)
		\left( \prod_{0\le i\le j<p}(z+x_i+x_j)(z+y_i+y_j) \right).
	\]
	Consequently, all monomials of $P_N(\vec x,\vec y,z)$ have total degree $\ge2p^2$. Furthermore, we see that
		\begin{multline}
				P_N(\vec x,\vec y,z) \\ = 
				\left( \prod_{0\le i<j<p}(x_i-x_j)(y_i-y_j) \right)
				\left( \prod_{0\le i\le j<p}(z+x_i+x_j)(z+y_i+y_j) \right)
				n^{-p^2}C(n)
				\left( 1+O\left(n^{-1}  \right) \right)
				\label{eq:P_N_expr2}
		\end{multline}
		for some unknown function $C(n)$ as $n\to\infty$. This function $C(n)$ can be determined by comparing
		the coefficient of $\prod_{i=0}^{p-1}x_i^{2i+1}y_i^{2i+1}$ in (\ref{eq:P_N_expr1}) and (\ref{eq:P_N_expr2}).
		In this way we obtain
		\begin{equation}\label{eq:C_n_identity}
		n^{-p^2}\sum_{\vec a\ge\vec 0}q_{\vec a,\vec J,\vec J,\vec 0}(n^{-1})G_{s,\vec a}(n)
		=4^pn^{-p^2}C(n)\left( 1+O(n^{-1}) \right),\qquad n\to\infty,
		\end{equation}
		where $\vec J=(1,3,\dots,2p-1)$.
		Since 
		\begin{equation}\label{eq:G_s,a_asymptotics}
		G_{s,\vec a}(n)=-sg_{s-1,\vec a}(n)+3\binom{s}{2}g_{s-2,\vec a}(n)+ O(g_{s-3,\vec a}(n)),
		\qquad n\to\infty,
		\end{equation}
		by (\ref{eq:G_a_expansion}), we see by (\ref{eq:sum_q_g_asymptotics}) that 
		\[
		\sum_{\vec a\ge\vec 0}q_{\vec a,\vec J,\vec J,\vec 0}(n^{-1})G_{s,\vec a}(n)
		= -\sum_{\vec a\ge\vec 0}q_{\vec a,\vec J,\vec J,\vec 0}(n^{-1})
			\left(sg_{s-1,\vec a}(n)-3\binom{s}{2}g_{s-2,\vec a}(n) \right)
			+O\left(n^{s/2-1}+n^{p/2-p^2}\log n  \right)
		\]
		as $n\to\infty$.
		Noting that
		\[
		q_{\vec a,\vec J,\vec J,\vec 0}(n^{-1})=
		\frac{2^p(-4)^{|\vec a|_1}}{\left( \prod_{i=0}^{p-1}(2i+1)! \right)^2}
		\det_{0\le i,j<p}\left( \frac{(2i+2j+2)!}{(i+j+1-a_i)!\ (2a_i)!} \right)
		+O(n^{-1}),
		\qquad n\to\infty,
		\]
		we further see by (\ref{eq:sum_q_g_asymptotics}) that
		\[
		\sum_{\vec a\ge\vec 0}q_{\vec a,\vec J,\vec J,\vec 0}(n^{-1})G_{s,\vec a}(n)=
		\frac{2^p}{\left( \prod\limits_{i=0}^{p-1}(2i+1)! \right)^2}
		\left( s\lambda_s n^{s/2}-3\binom{s}{2}\lambda_{s-1}n^{(s-1)/2}+\lambda_0
			+O\left( n^{s/2-1}+n^{p/2-p^2}\log n \right) \right)
		\]
		as $n\to\infty$, where
		\[
		\lambda_k=-\sum_{\vec a\ge\vec 0}(-4)^{|\vec a|_1}
		\det_{0\le i,j<p}\left(\frac{(2i+2j+2)!}{(i+j+1-a_i)!\ (2a_i)!}   \right)\omega_{k-1,\vec a},
		\qquad k>0,
		\]
		and
		\[
		\lambda_0 = -\frac{3}{2}\det_{0\le i,j<p}\left( \frac{(2i+2j+2)!}{(i+j+1)!} \right).
		\]
		Here, the constant $\lambda_0$ is of interest only in the case $s=1$ (it can be absorbed into the $O$-term otherwise),
		and comes from the asymptotic expansion of $g_{1,\vec 0}(n)$.
		
		Now, with the help of Equation~(\ref{eq:C_n_identity}) we can determine asymptotics for the function $C(n)$, which gives
		us asymptotics for $P_N(\vec x,\vec y,z)$ by Equation~(\ref{eq:P_N_expr2}), and finally also
		asymptotics for $D_N(\vec x,\vec y,z)$ by Equation~(\ref{eq:D_n_asymptotics}).

		The proof is now completed upon specialising to $x_i=y_i=i$ for $i=0,\dots,p-1$ and $z=1$ in the asymptotics for
		$D_N(\vec x,\vec y,z)$. For sake of convenience we finally note the identities
		\begin{align*}
		\left(\prod_{0\le i<j<p}(i-j)^2\right)\left(\prod_{0\le i\le j<p}(1+i+j)^2\right) &= \prod_{i=0}^{p-1}(2i+1)!^2, \\
		\det_{0\le i,j<p}\left( \frac{(2i+2j+2)!}{(i+j+1)!} \right) &= 2^{p^2}\prod_{i=0}^{p-1}(2i+1)!.
		\end{align*}
		The second identity can be proved by means of standard determinant evaluation techniques
		(see \cite{MR1701596} for details).
\end{proof}

Finally, we can state and prove the main result of this paper.
\subsection*{Proof of Theorem~\ref{thm:sth-moment-asymptotics}}
		Replacing $M_{2n}^{(p)}$ and the sum in Equation~(\ref{eq:wm_with_wall_sth_moment_def})
		with their asymptotic expansions as given by  Lemma~\ref{lem:wm_with_wall_asymptotic_total} and
	Lemma~\ref{lem:wm_with_wall_asymptotic_for_sum} we see that
	\[
	\E H_{n,p}^{s}=s\kappa_{s}^{(p)}n^{s/2}-3\binom{s}{2}\kappa_{s-1}^{(p)}n^{(s-1)/2}-\frac{3}{2}
	+O\left( n^{s/2-1}+n^{p/2-p^2}\log n \right)
	,\qquad n\to\infty,
	\]
	where, for $k>0$, 
	\[
	\kappa_k^{(p)}=-\frac{1}{M_p}\sum_{\vec a\ge\vec 0}(-4)^{|\vec a|_1}
	\det_{0\le i,j<p}\left( \frac{(2i+2j+2)!}{(i+j+1-a_i)!\ (2a_j)!} \right)\omega_{k-1,\vec a}.
	\]
	The quantity $\omega_{k-1,\vec a}$ has already been defined in Lemma~\ref{lem:g_a_asymptotics}.
	
	In order to prove the integral representation for $\kappa_s^{(p)}$ when $s\neq p$,
	where we have $\omega_{k-1,\vec a}=\frac{1}{2}\Gamma\left( \frac{k}{2}+|\vec a|_1 \right)Z_{2\vec a}\left( \frac{k}{2}+|\vec a|_1 \right)$, 
	we consider the more general expression
		\begin{align*}
				\kappa^{(p)}(z) &= -\frac{1}{2 M_p}  
				\sum_{\vec a\ge \vec 0}
				\det_{0\le i,j<p}\left( \frac{(2i+2j+2)!(-4)^{a_i}}{(i+j+1-a_i)!(2a_i)!} \right)
				\Gamma\left( z+|\vec a|_1 \right)Z_{2\vec a}\left( z+|\vec a|_1 \right) \\
				&=
				-\frac{\pi^z}{2 M_p}
				\int_0^\infty t^{z-1}
				\left(
					\det\left( \sum_{a\ge 0}\frac{(2i+2j+2)!(t/\pi)^a}{(i+j+1-a)!(2a)!}\vartheta_{2a}(t) \right)
					-M_p
				\right)dt
		\end{align*}
		for $\Re z$ sufficiently large.
		Here, the second line is a direct consequence of Equation~(\ref{eq:Z_riemann}).
		The reciprocity relation~(\ref{eq:vartheta_a_identity}) followed by the change of variables $t\mapsto t^{-1}$ then shows that
		\[
		\kappa^{(p)}(z) = \frac{\pi^z}{2}
		\int_0^\infty t^{-z-1}\left(1 - \frac{t^{p^2+p/2}}{(-\pi)^{p^2}}\frac{T_p(t)}{M_p} \right)dt.
		\]
		Asymptotics for $\vartheta_{2a}(t)$ for $t\to 0$ and $t\to\infty$ as given in the proof of
		Lemma~\ref{lem:Z_vertical_growth} then show that this last integral is convergent for $\Re z>0$.
		The representation for $s\neq p$ is now proved upon observing that, by definition, we have 
		$\kappa^{(p)}\left( \frac{s}{2} \right)=\kappa_s^{(p)}$.

		Now, consider the case $s=p$. Here, we have
		\[
		\omega_{p-1,\vec a}=\frac{1}{2}
		\lim_{z\to p/2}\left(Z_{2\vec a}\left( z+|\vec a|_1 \right)\Gamma(z+|\vec a|_1)
		-\frac{\pi^{p/2}\left( \prod_{i=0}^{p-1}\frac{(2a_i)!}{4^{a_i}a_i!} \right)}
			  {\left( z-\frac{p}{2} \right)}
		\right),
		\]
		which shows that, as in the other case,
		\[
		-\frac{1}{M_p}\sum_{\vec a\ge \vec 0}
			\det_{0\le i,j<p}\left( \frac{(2i+2j+2)!(-4)^{a_i}}{(i+j+1-a_i)!(2a_i)!} \right)
			\omega_{p-1,\vec a}=
			\lim_{z\to p/2}
			\kappa^{(p)}(z) = \kappa^{(p)}\left( \frac{p}{2} \right).
		\]
		In this last calculation, we have, after interchanging the order of the limit and the sum,
		applied the results obtained in the case $s\neq p$.
		This proves Theorem~\ref{thm:sth-moment-asymptotics}.

	\begin{remark}
		\label{rem:complex_moments}
		It can be shown that Theorem~\ref{thm:sth-moment-asymptotics} is even valid for $s\in\C$, $\Re(s)>0$.
		The proof of this more general result is the same as for our theorem except for two small changes which we are
		going to address now.

		In the proof of Lemma~\ref{lem:wm_with_wall_asymptotic_for_sum} we defined the functions $G_{s,\vec a}(n)$
		(see Equation~(\ref{eq:G_def})). For $s\in\N$ the asymptotics~(\ref{eq:G_s,a_asymptotics}) for $G_{s,\vec a}(n)$
		were easily found by the expansion~(\ref{eq:G_a_expansion}). This is not possible for $s\in\C\setminus\N$.
		In order to prove the asymptotics~(\ref{eq:G_s,a_asymptotics}) in that case we note that
		(see Equation~(\ref{eq:G_def}))
		\[
		(h-1)^s-h^s = (h+1)^s\left( \left( 1-\frac{2}{h+1} \right)^s-\left( 1-\frac{1}{h+1} \right)^s \right).
		\]
		The term for $h=1$ in (\ref{eq:G_def}) is seen to be negligible due to summation over $\vec a\ge\vec 0$
		(see the discussion of the function $\Omega_k(n)$ following Equation~(\ref{eq:G_a_expansion}) in the proof of
		Lemma~\ref{lem:wm_with_wall_asymptotic_for_sum}).
		For $h\ge 2$, we can use the binomial series expansion in the expression above and finally obtain the
		asymptotics~(\ref{eq:G_s,a_asymptotics}).

		The second change concerns Lemma~\ref{lem:g_a_asymptotics}, which has to be generalised to $k\in\C$. But this
		makes no difficulties.
\end{remark}
	
\section{Proof of Theorem~\ref{thm:centrallimitlaw}: A central limit law}

We are now going to derive the claimed asymptotics for the cumulative
distribution function of the random variable ``height'' on the set of $p$-watermelons with length $2n$ with wall, i.e., 
\[ F_n(h) = \P\left\{ H_{n,p}\le h \right\} = \frac{M_{2n,h+1}^{(p)}}{M_{2n}^{(p)}} \]
for the range $h+2=t\sqrt{n}$, where $t\in(0,\infty)$.
		
The result can be proved in pretty much the same way as Theorem~\ref{thm:sth-moment-asymptotics}. Therefore, 
		we only give a rather brief account of the proof,
		and refer to Lemma~\ref{lem:wm_with_wall_asymptotic_for_sum} for the details.
		
		Instead of 
		the exact expression~(\ref{eq:exact_M_2n_h}) for $M_{2n,h}^{(p)}$ we consider the more general 
		quantity
		\begin{equation}
				F_n(h;\vec x,\vec y,z) =
				\binom{2n}{n}^p
				\det_{0\le i,j<p}\left( \sum_{m\in(h+2)\Z}\frac{\binom{2n}{n+m+x_i-y_j}}{\binom{2n}{n}}-\frac{\binom{2n}{n+m-z-x_i-y_j}}{\binom{2n}{n}} \right).
				\label{eq:cdf_generalisation}
		\end{equation}
		Again, we find the polynomial
		\[
		Q_N(\vec x,\vec y,z)=
		\det_{0\le i,j<p}\left( \sum_{u=0}^{2N} \frac{(y_j-x_i)^{2u}-(z+x_i+y_j)^{2u}}{n^u}\sum_{m\in(h+2)\Z}T_{2u;N}\left( \frac{m}{\sqrt n},n \right)e^{-m^2/n} \right),
		\]
		such that
		\[
		F_n(h;\vec x,\vec y,z) = \binom{2n}{n}^p \left( Q_N(\vec x,\vec y,z)
		+O\left(n^{-2N-1}\sum_{m\in(h+2)\Z}e^{-m^2/n}\right)\right),
		\]
		where $N$ can be chosen arbitrarily large and $T_{2u;N}$ being defined by (\ref{eq:T_def}).
		The polynomial $Q_N(\vec x,\vec y,z)$ is seen to be divisible by
		\[
		\left( \prod_{0\le i<j<p}(x_i-x_j)(y_i-y_j) \right)\left( \prod_{0\le i\le j<p}(z+x_i+x_j)(z+y_i+y_j) \right)
		\]
		since the determinant in the definition of $Q_N(\vec x,\vec y,z)$ vanishes whenever
		$x_i=x_j$ or $y_i=y_j$ for some $i\neq j$ or $x_i=-z-x_j$ or $y_i=-z-y_j$ for some $i$ and $j$.
		Hence, 
		\[
		Q_N(\vec x,\vec y,z)
		=
		\left( \prod_{0\le i<j<p}(x_i-x_j)(y_i-y_j) \right)\left( \prod_{0\le i\le j<p}(z+x_i+x_j)(z+y_i+y_j) \right)
		C(t)\left( 1+O(n^{-1}) \right)
		\]
		as $n\to\infty$ for some unknown constant $C(t)$.
		Now, we are going to determine asymptotics for $C(t)$ as $n\to\infty$. This task can be accomplished by comparing the coefficients
		of the monomial $\prod_{0\le i<p}x_i^{2i+j}y_i^{2i+1}$ in the expression above and the defining expression for $Q_N(\vec x,\vec y,z)$.
		We obtain
		\[
		4^pC(t)\left( 1+O(n^{-1}) \right)
		=
		(-2)^{p}n^{-p^2}
		\det_{0\le i,j<p}\left(\binom{2i+2j+2}{2i+1}\sum_{m\in \Z}T_{2i+2j+2;N}\left( mt,n \right)e^{-(mt)^2}  \right).
		\]
		Recalling the definition of the functions $T_{2a;N}(w,n)$ (see Equation~(\ref{eq:T_def})), our attention is drawn to sums of the form
		\[ (2a)!\sum_{n=-\infty}^{\infty}\phi_{2a}(mt)e^{-(mt)^2},\qquad a\in\N, \]
		where the polynomials $\phi_{2a}$ are defined by (\ref{eq:phi_s_def}).

		Now, rewriting the reciprocity relation~(\ref{eq:vartheta_a_identity}) as 
		\[
		\vartheta_{2a}\left( \frac{1}{y} \right)=y^{a+1/2}\pi^a\sum_{n=-\infty}^{\infty}(2a)!\phi_{2a}\left( n\sqrt{\pi y} \right)e^{-n^2\pi y},
		\]
		we see that
		\[
		(2a)!\sum_{n=-\infty}^{\infty}\phi_{2a}(mt)e^{-(mt)^2} = \frac{\sqrt{\pi}}{t^{2a+1}}\vartheta_{2a}\left( \frac{\pi}{t^2} \right).
		\]
		From the asymptotics for $\vartheta_{2a}(t)$ as given in the proof of Lemma~\ref{lem:Z_vertical_growth} we deduce that
		\begin{align*}
		\frac{\sqrt{\pi}}{t^{2a+1}}\vartheta_{2a}\left( \frac{\pi}{t^2} \right) &= \textrm{const}+O(t^{-M}),\qquad t\to \infty, \\
		\frac{\sqrt{\pi}}{t^{2a+1}}\vartheta_{2a}\left( \frac{\pi}{t^2} \right) &= [a=0]\frac{\sqrt\pi}{t}+O\left(t^{-M} \right),\qquad t\to 0,
		\end{align*}	
		for all $M>0$.
		Consequently, we obtain for $h+2=t\sqrt n$, where $t$ is fixed,
		\[
		\sum_{m\in\Z}T_{2a;N}\left(mt,n \right)e^{-(mt)^2}
		= 
		\frac{\sqrt{\pi}}{t^{2a+1}}\vartheta_{2a}\left( \frac{\pi}{t^2} \right)
		+O\left( \frac{1}{nt} \right),\qquad n\to\infty.
		\]
		Note that the constant implied by the $O$-term can be chosen independent of $t$.
		
		Theorem~\ref{thm:centrallimitlaw} is now proved upon substituting these asymptotics for our sums appearing in
		the expression for $C(t)$ above, taking out some factors, specialising to $x_j=y_j=1$, and dividing by $M_{2n}^{(p)}$.

\appendix

\bibliographystyle{plain}
\bibliography{pmelons}

\end{document}

%% file: watermelons_fig.tex
\input{./pathlate.sty}


$$
\Koordinatenachsen(19,15)(0,0)
\Pfad(0,0),334334443333443444\endPfad
\Pfad(0,2),333434334433434444\endPfad
\Pfad(0,4),333433334443434444\endPfad
\Pfad(0,6),333333344433444444\endPfad
\PfadDicke{.5pt}
\SPfad(0,13),1111111111111111111\endSPfad
\DickPunkt(0,0)
\DickPunkt(0,2)
\DickPunkt(0,4)
\DickPunkt(0,6)
\DickPunkt(18,0)
\DickPunkt(18,2)
\DickPunkt(18,4)
\DickPunkt(18,6)
\DuennPunkt(0,0)
\DuennPunkt(2,0)
\DuennPunkt(4,0)
\DuennPunkt(6,0)
\DuennPunkt(8,0)
\DuennPunkt(10,0)
\DuennPunkt(12,0)
\DuennPunkt(14,0)
\DuennPunkt(16,0)
\DuennPunkt(18,0)
\DuennPunkt(1,1)
\DuennPunkt(3,1)
\DuennPunkt(5,1)
\DuennPunkt(7,1)
\DuennPunkt(9,1)
\DuennPunkt(11,1)
\DuennPunkt(13,1)
\DuennPunkt(15,1)
\DuennPunkt(17,1)
\DuennPunkt(0,2)
\DuennPunkt(2,2)
\DuennPunkt(4,2)
\DuennPunkt(6,2)
\DuennPunkt(8,2)
\DuennPunkt(10,2)
\DuennPunkt(12,2)
\DuennPunkt(14,2)
\DuennPunkt(16,2)
\DuennPunkt(18,2)
\DuennPunkt(1,3)
\DuennPunkt(3,3)
\DuennPunkt(5,3)
\DuennPunkt(7,3)
\DuennPunkt(9,3)
\DuennPunkt(11,3)
\DuennPunkt(13,3)
\DuennPunkt(15,3)
\DuennPunkt(17,3)
\DuennPunkt(0,4)
\DuennPunkt(2,4)
\DuennPunkt(4,4)
\DuennPunkt(6,4)
\DuennPunkt(8,4)
\DuennPunkt(10,4)
\DuennPunkt(12,4)
\DuennPunkt(14,4)
\DuennPunkt(16,4)
\DuennPunkt(18,4)
\DuennPunkt(1,5)
\DuennPunkt(3,5)
\DuennPunkt(5,5)
\DuennPunkt(7,5)
\DuennPunkt(9,5)
\DuennPunkt(11,5)
\DuennPunkt(13,5)
\DuennPunkt(15,5)
\DuennPunkt(17,5)
\DuennPunkt(0,6)
\DuennPunkt(2,6)
\DuennPunkt(4,6)
\DuennPunkt(6,6)
\DuennPunkt(8,6)
\DuennPunkt(10,6)
\DuennPunkt(12,6)
\DuennPunkt(14,6)
\DuennPunkt(16,6)
\DuennPunkt(18,6)
\DuennPunkt(1,7)
\DuennPunkt(3,7)
\DuennPunkt(5,7)
\DuennPunkt(7,7)
\DuennPunkt(9,7)
\DuennPunkt(11,7)
\DuennPunkt(13,7)
\DuennPunkt(15,7)
\DuennPunkt(17,7)
\DuennPunkt(0,8)
\DuennPunkt(2,8)
\DuennPunkt(4,8)
\DuennPunkt(6,8)
\DuennPunkt(8,8)
\DuennPunkt(10,8)
\DuennPunkt(12,8)
\DuennPunkt(14,8)
\DuennPunkt(16,8)
\DuennPunkt(18,8)
\DuennPunkt(1,9)
\DuennPunkt(3,9)
\DuennPunkt(5,9)
\DuennPunkt(7,9)
\DuennPunkt(9,9)
\DuennPunkt(11,9)
\DuennPunkt(13,9)
\DuennPunkt(15,9)
\DuennPunkt(17,9)
\DuennPunkt(0,10)
\DuennPunkt(2,10)
\DuennPunkt(4,10)
\DuennPunkt(6,10)
\DuennPunkt(8,10)
\DuennPunkt(10,10)
\DuennPunkt(12,10)
\DuennPunkt(14,10)
\DuennPunkt(16,10)
\DuennPunkt(18,10)
\DuennPunkt(1,11)
\DuennPunkt(3,11)
\DuennPunkt(5,11)
\DuennPunkt(7,11)
\DuennPunkt(9,11)
\DuennPunkt(11,11)
\DuennPunkt(13,11)
\DuennPunkt(15,11)
\DuennPunkt(17,11)
\DuennPunkt(0,12)
\DuennPunkt(2,12)
\DuennPunkt(4,12)
\DuennPunkt(6,12)
\DuennPunkt(8,12)
\DuennPunkt(10,12)
\DuennPunkt(12,12)
\DuennPunkt(14,12)
\DuennPunkt(16,12)
\DuennPunkt(18,12)
\DuennPunkt(1,13)
\DuennPunkt(3,13)
\DuennPunkt(5,13)
\DuennPunkt(7,13)
\DuennPunkt(9,13)
\DuennPunkt(11,13)
\DuennPunkt(13,13)
\DuennPunkt(15,13)
\DuennPunkt(17,13)
\DuennPunkt(0,14)
\DuennPunkt(2,14)
\DuennPunkt(4,14)
\DuennPunkt(6,14)
\DuennPunkt(8,14)
\DuennPunkt(10,14)
\DuennPunkt(12,14)
\DuennPunkt(14,14)
\DuennPunkt(16,14)
\DuennPunkt(18,14)
\Label\lu{0}(0,0)
\Label\u{18}(18,0)
\Label\l{13}(0,13)
\hskip11cm
$$


%% file: pmelons.bbl
\begin{thebibliography}{10}

\bibitem{MR1848256}
Philippe Biane, Jim Pitman, and Marc Yor.
\newblock Probability laws related to the {J}acobi theta and {R}iemann zeta
  functions, and {B}rownian excursions.
\newblock {\em Bull. Amer. Math. Soc. (N.S.)}, 38(4):435--465 (electronic),
  2001.

\bibitem{MR2022577}
Nicolas Bonichon and Mohamed Mosbah.
\newblock Watermelon uniform random generation with applications.
\newblock {\em Theoret. Comput. Sci.}, 307(2):241--256, 2003.
\newblock Random generation of combinatorial objects and bijective
  combinatorics.

\bibitem{MR692105}
Pierrette Cassou-Nogu{\`e}s.
\newblock Prolongement de certaines s\'eries de {D}irichlet.
\newblock {\em Amer. J. Math.}, 105(1):13--58, 1983.

\bibitem{MR0505710}
N.~G. de~Bruijn, D.~E. Knuth, and S.~O. Rice.
\newblock The average height of planted plane trees.
\newblock In {\em Graph theory and computing}, pages 15--22. Academic Press,
  New York, 1972.

\bibitem{MR2278751}
Marc de~Crisenoy.
\newblock Values at {$T$}-tuples of negative integers of twisted multivariable
  zeta series associated to polynomials of several variables.
\newblock {\em Compos. Math.}, 142(6):1373--1402, 2006.

\bibitem{deCrisenoy}
Marc de~Crisenoy and Driss Essouabri.
\newblock Relations between values at $t$-tuples of negative integers of
  multivariate twisted zeta series associated to polynomials of several
  variables.
\newblock Preprint, arXiv:math.NT/0505568, 2007.

\bibitem{MR0148397}
Freeman~J. Dyson.
\newblock A {B}rownian-motion model for the eigenvalues of a random matrix.
\newblock {\em J. Mathematical Phys.}, 3:1191--1198, 1962.

\bibitem{MR1384738}
J.~W. Essam and A.~J. Guttmann.
\newblock Vicious walkers and directed polymer networks in general dimensions.
\newblock {\em Phys. Rev. E (3)}, 52(6, part A):5849--5862, 1995.

\bibitem{watermelons:withoutwall}
Thomas Feierl.
\newblock The height and range of watermelons without wall.
\newblock Preprint, arXiv:math.CO/0806.0037.

\bibitem{MR751710}
Michael~E. Fisher.
\newblock Walks, walls, wetting, and melting.
\newblock {\em J. Statist. Phys.}, 34(5-6):667--729, 1984.

\bibitem{MR1337752}
Philippe Flajolet, Xavier Gourdon, and Philippe Dumas.
\newblock Mellin transforms and asymptotics: harmonic sums.
\newblock {\em Theoret. Comput. Sci.}, 144(1-2):3--58, 1995.
\newblock Special volume on mathematical analysis of algorithms.

\bibitem{MR2747559}
Peter~J. Forrester, Satya~N. Majumdar, and Gr{\'e}gory Schehr.
\newblock Non-intersecting {B}rownian walkers and {Y}ang-{M}ills theory on the
  sphere.
\newblock {\em Nuclear Phys. B}, 844(3):500--526, 2011.

\bibitem{Fulmek}
Markus Fulmek.
\newblock Asymptotics of the average height of $2$-watermelons with a wall.
\newblock {\em Electronic J. Combin.}, 14(1):R64, 20 pp. (electronic), 2007.

\bibitem{gesselviennot}
Ira Gessel and G\'erard Viennot.
\newblock Determinants, paths, and plane partitions.
\newblock Preprint, available at
  http://people.brandeis.edu/~gessel/homepage/papers/pp.pdf, 1989.

\bibitem{MR1092920}
Ira~M. Gessel and Doron Zeilberger.
\newblock Random walk in a {W}eyl chamber.
\newblock {\em Proc. Amer. Math. Soc.}, 115(1):27--31, 1992.

\bibitem{gillet}
Florent Gillet.
\newblock Asymptotic behaviour of watermelons.
\newblock Preprint, arXiv:math.PR/0307204, 2003.

\bibitem{MR1651492}
Anthony~J. Guttmann, Aleksander~L. Owczarek, and Xavier~G. Viennot.
\newblock Vicious walkers and {Y}oung tableaux. {I}. {W}ithout walls.
\newblock {\em J. Phys. A}, 31(40):8123--8135, 1998.

\bibitem{katori2}
Makoto Katori, Minami Izumi, and Naoki Kobayashi.
\newblock Maximum distribution of bridges of noncolliding brownian paths.
\newblock {\em Phys. Rev. E}, 78:051102 (electronic), 2008.

\bibitem{katori}
Makoto Katori, Minami Izumi, and Naoki Kobayashi.
\newblock Two {B}essel bridges conditioned never to collide, double {D}irichlet
  series, and {J}acobi theta function.
\newblock {\em J. Stat. Phys.}, 131(6):1067--1083, 2008.

\bibitem{MR2029612}
Makoto Katori and Hideki Tanemura.
\newblock Functional central limit theorems for vicious walkers.
\newblock {\em Stoch. Stoch. Rep.}, 75(6):369--390, 2003.

\bibitem{MR1701596}
C.~Krattenthaler.
\newblock Advanced determinant calculus.
\newblock {\em S\'em. Lothar. Combin.}, 42:Art. B42q, 67 pp. (electronic),
  1999.
\newblock The Andrews Festschrift (Maratea, 1998).

\bibitem{Kratt_Wall_Contact}
Christian Krattenthaler.
\newblock Watermelon configurations with wall interaction: exact and asymptotic
  results.
\newblock {\em J. Physics Conf. Series}, 42:179--212, 2006.

\bibitem{MR1801472}
Christian Krattenthaler, Anthony~J. Guttmann, and Xavier~G. Viennot.
\newblock Vicious walkers, friendly walkers and {Y}oung tableaux. {II}. {W}ith
  a wall.
\newblock {\em J. Phys. A}, 33(48):8835--8866, 2000.

\bibitem{MR1889901}
Ekkehard Kr{\"a}tzel.
\newblock {\em Analytische Funktionen in der Zahlentheorie}, volume 139 of {\em
  Teubner-Texte zur Mathematik [Teubner Texts in Mathematics]}.
\newblock B. G. Teubner, 2000.

\bibitem{Liechty}
Karl Liechty.
\newblock The limiting distribution of the maximal height of the outermost path
  of nonintersecting brownian excursions and discrete gaussian orthogonal
  polynomials.
\newblock Preprint, available at http://arxiv.org/abs/1111.4239, 2011.

\bibitem{MR0335313}
Bernt Lindstr{\"o}m.
\newblock On the vector representations of induced matroids.
\newblock {\em Bull. London Math. Soc.}, 5:85--90, 1973.

\bibitem{MR2129906}
Madan~Lal Mehta.
\newblock {\em Random matrices}, volume 142 of {\em Pure and Applied
  Mathematics (Amsterdam)}.
\newblock Elsevier/Academic Press, Amsterdam, third edition, 2004.

\bibitem{MR554084}
Sri~Gopal Mohanty.
\newblock {\em Lattice path counting and applications}.
\newblock Academic Press [Harcourt Brace Jovanovich Publishers], New York,
  1979.
\newblock Probability and Mathematical Statistics.

\bibitem{MR1832080}
A.~L. Owczarek, J.~W. Essam, and R.~Brak.
\newblock Scaling analysis for the adsorption transition in a watermelon
  network of $n$ directed non-intersecting walks.
\newblock {\em J. Statist. Phys.}, 102(3-4):997--1017, 2001.

\bibitem{MR0219440}
A.~R{\'e}nyi and G.~Szekeres.
\newblock On the height of trees.
\newblock {\em J. Austral. Math. Soc.}, 7:497--507, 1967.

\bibitem{schehr}
Gr\'egory Schehr, Satya~N. Majumdar, Alain Comtet, and Julien Randon-Furling.
\newblock Exact distribution of the maximal height of $p$ vicious walkers.
\newblock {\em Physical Review Letters}, 101(15):150601, 2008.

\bibitem{MR882550}
E.~C. Titchmarsh.
\newblock {\em The Theory of the Riemann Zeta-function}.
\newblock The Clarendon Press Oxford University Press, 2 edition, 1986.
\newblock Edited and with a preface by D. R. Heath-Brown.

\bibitem{MR2326237}
Craig~A. Tracy and Harold Widom.
\newblock Nonintersecting {B}rownian excursions.
\newblock {\em Ann. Appl. Probab.}, 17(3):953--979, 2007.

\bibitem{MR0178117}
E.~T. Whittaker and G.~N. Watson.
\newblock {\em A course of modern analysis. {A}n introduction to the general
  theory of infinite processes and of analytic functions: with an account of
  the principal transcendental functions}.
\newblock Fourth edition. Reprinted. Cambridge University Press, New York,
  1962.

\end{thebibliography}
